\documentclass[reqno, 11pt]{amsart}
\usepackage[utf8]{inputenc}
\usepackage[all]{xy}
\usepackage{enumerate}
\usepackage[english]{babel}
\usepackage[dvipsnames]{xcolor}
\usepackage{graphicx,pstricks,pst-plot}
\usepackage[colorlinks=true, linkcolor=MidnightBlue, citecolor=purple, urlcolor=green]{hyperref}
\usepackage{setspace}
\usepackage{cancel}
\usepackage[left=2.5cm, right=2.5cm, bottom=2cm]{geometry} 
\usepackage{amsmath, amssymb, amsthm, epsfig, mathtools}
\usepackage{latexsym}
\usepackage[nameinlink,capitalise]{cleveref}
\usepackage{url}
\usepackage{bbm}

\newcommand{\ld}{\lambda}
\newcommand{\Ld}{\Lambda}
\newcommand{\al}{\alpha}
\newcommand{\veps}{\varepsilon}

\DeclareMathOperator{\Int}{Int}
\DeclareMathOperator{\loc}{loc}
\DeclareMathOperator{\dist}{dist}

\newcommand{\bA}{\boldsymbol{A}}

\DeclareMathOperator{\meas}{\mathrm{meas}}
\DeclareMathOperator{\supp}{\mathrm{supp}}
\DeclareMathOperator{\lin}{\mathrm{lin}}
\DeclareMathOperator{\Nlin}{\mathrm{Nlin}}

\providecommand{\norm}[1]{\lVert#1\rVert}

\providecommand{\inn}[1]{\langle#1\rangle}

\newcommand{\bC}{\mathbf{C}}
\newcommand{\bB}{\mathbf{B}}
\newcommand{\bP}{\mathbf{P}}
\newcommand{\mc}{\mathcal}
\newcommand{\A}{\mc{A}}

\newcommand{\M}{\mc{M}}

\newcommand{\R}{\mathbb{R}}
\newcommand{\N}{\mathbb{N}}

\def\today{\ifcase\month\or
  January\or February\or March\or April\or May\or June\or
  July\or August\or September\or October\or November\or December\fi
  \space\number\day, \number\year}

\newtheorem{theorem}{Theorem}[section]
\newtheorem{lemma}[theorem]{Lemma}
\newtheorem{proposition}[theorem]{Proposition}
\newtheorem{corollary}[theorem]{Corollary}
\theoremstyle{definition}
\newtheorem{definition}[theorem]{Definition}

\theoremstyle{remark}
\newtheorem{remark}[theorem]{Remark}
\newtheorem{assum}{Assumption}

\newtheorem{property}{Property}

\begin{document}

\title{Stabilization and control of the semilinear\\ hinged plate equation}
\author[]{Cristóbal Loyola}
\date{October 30, 2026}
\subjclass[2020]{35B40, 35B41, 93B05, 93D20, 35B60, 35L75}
% 35B40 	Asymptotic behavior of solutions to PDEs 
% 35B41 	Attractors
% 93B05 	Controllability
% 93D20 	Asymptotic stability in control theory
% 35B60     Continuation and prolongation of solutions to PDEs
% 35L75 	Higher-order nonlinear hyperbolic equations
\keywords{damped plate equation, exponential stabilization, compact global attractor, internal control, unique continuation, semilinear plate equation}
\address{Sorbonne Université, Université Paris Cité, CNRS, Laboratoire Jacques-Louis Lions, LJLL, F-75005 Paris, France}
\email{cristobal.loyola@sorbonne-universite.fr}
\allowdisplaybreaks
\numberwithin{equation}{section}

\begin{abstract}
    In this article we prove semiglobal stabilization and exact controllability results for semilinear plate equations with hinged boundary conditions and analytic nonlinearity, when the damping or control acts on a region where the linear Schrödinger equation is observable. At the core of these results lies a new unique continuation property for the nonlinear plate equation, which significantly relaxes the geometric conditions required for such a property to hold. This property is obtained by combining recent results on propagation of analyticity in time and unique continuation for linear plate operators. More broadly, our approach exploits the linear observability of the plate equation to establish both stabilization and control results. First, we prove exponential decay of the nonlinear energy under a defocusing assumption on the nonlinearity. Second, under a weaker asymptotic assumption on the nonlinearity, we prove semiglobal exact control by analyzing control properties inside the compact attractor provided by the dynamics of the damped equation.
\end{abstract}

\maketitle
\setcounter{tocdepth}{1}
\tableofcontents

% % REQUIRED
% \begin{keywords}
% damped plate equation, exponential stabilization, compact global attractor, internal control, unique continuation, semilinear plate equation
% \end{keywords}

% % REQUIRED
% \begin{MSCcodes}
% 35B40, 35B41, 93B05, 93D20, 35B60, 35L75
% \end{MSCcodes}

\section{Introduction} 

Let $(\M, g)$ be a smooth compact connected Riemannian manifold of dimension $d\leq 3$ with (possibly empty) smooth boundary $\partial\M$. In this article we consider the following nonlinear plate equation
\begin{align}\label{eq:nlp}
    \left\{\begin{array}{rl}
        \partial_t^2 u+\Delta_g^2 u+\beta u+f(x, u)=0   &\ (t, x)\in [0, T]\times \Int(\M),\\
        u_{|_{\partial\M}}=\Delta u_{|_{\partial\M}}=0      &\ (t, x)\in [0, T]\times\partial\M,\\
        (u, \partial_t u)(0)=(u_0, u_1) &\ x\in \M,
    \end{array}\right.
\end{align}
with $(u_0, u_1)\in H_D^2\times L^2(\M)$ where $H_D^2:=H^2(\M)\cap H_0^1(\M)$, and $\Delta_g^2$ is the square of the Laplace-Beltrami operator equipped with hinged boundary conditions whenever $\partial\M\neq\emptyset$. Here the nonlinearity $f:\M\times \R\to \R$ is assumed to be smooth in the first variable and analytic in the second one. The constant $\beta\geq 0$ is so that we have the Poincaré-type inequality, we thus may require $\beta>0$ when $\partial\M=\emptyset$.

Equation \eqref{eq:nlp} belongs to the family of fourth-order nonlinear wave, or nonlinear plate, equations. Its physical interpretation is clearest for $d=2$ and the flat metric: $\partial_t^2+\Delta^2$ is the classical linear model for the transverse deflection $u(t,x)$ of a thin elastic plate, the biharmonic term encoding its flexural (bending) rigidity, while $\beta u+f(x,u)$ represent a pointwise restoring force, typically that of a possibly nonlinear elastic foundation on which the plate rests. Among the physically relevant nonlinearities:
\begin{enumerate}
    \item\label{ex:NL} The cubic nonlinearity $f(x,u)=k(x)|u|^2u$, where the sign of the stiffness $k(x)$ governs its behavior: for $k(x)>0$ the foundation is hardening, stiffening as the deflection grows and thus contributing to stability, while for $k(x)<0$ it is softening, weakening with amplitude and acting against it \cite[Section~2.7.2]{AEHY19}.
    \item The nonlinearity $f(x,u)=\tfrac{1}{2}\big(u+\sqrt{u^2+\veps}-\sqrt{\veps}\big)$ for small $\veps>0$, which stands for an analytic regularization of the positive part $u\mapsto\max\{u,0\}=u^+$, models a restoring force acting in one direction of displacement, see \cite{LM90}.
\end{enumerate}
The boundary conditions $u_{|_{\partial\M}}=\Delta u_{|_{\partial\M}}=0$ are the hinged, or simply supported, conditions: the edge is held at zero displacement but left free to rotate, so that no bending moment is transmitted across it.

We are interested in the geometric conditions on a nonempty open set $\omega\subset \M$ under which \eqref{eq:nlp} can be stabilized by a damping (local friction) or controlled by an input (a distributed transverse load) localized in $\omega$. Our starting point is that, under hinged boundary conditions, the plate operator can be decomposed as two Schrödinger-type operators $\partial_t^2+\Delta_g^2=-(i\partial_t-\Delta_g)(i\partial_t+\Delta_g)$. This property was exploited by Lebeau \cite[Section 5]{Leb92}, where observability and control properties are transferred from the linear Schrödinger into the plate equation. In this article we take the next step on transferring the geometries allowed in the linear case into a nonlinear setting. Our results hold whenever the damping or control acts on a nonempty open set $\omega\subset\M$ satisfying the following assumption:

\begin{assum}\label{assumOBS}
    The Schr\"odinger equation is observable in $L^2$ from $\omega$ in time $T>0$: there exists $C=C(T,\omega)>0$ such that for any $v_0\in L^2(\M)$ it holds
    \begin{align*}
        \norm{v_0}_{L^2(\M)}^2\leq C\int_0^T \norm{\mathbbm{1}_\omega v(t)}_{L^2(\M)}^2dt,
    \end{align*}
    where $v$ solves $i\partial_t v+\Delta_g v=0$ in $(0,T)\times\Int(\M)$ with $v_{|\partial\M}=0$, and $v(0)=v_0$.
\end{assum}

A sharp geometric characterization of the sets $\omega$ for which the Schr\"odinger equation is observable remains an open question. In the following physically relevant situations, \cref{assumOBS} is known to be true for any $T>0$: 
\begin{enumerate}
    \item $(\M, g)$ is a compact Riemannian manifold, with or without boundary, and $\omega$ satisfies the \emph{Geometric Control Condition} (GCC); see Lebeau \cite{Leb92}.
    \item $(\M,g)=((0,1)^d,\text{Euclid})$ with $d\leq 3$, $\omega$ is any nonempty open set. This was first proved by Jaffard for $d=2$ \cite{Jaf90} and Komornik \cite{K:92} for other dimensions (actually, directly for the plate equation). Other proofs were given later by Burq-Zworski \cite{BZ:04}, Anantharaman-Maci\`a \cite{AM:14} on the torus $\mathbb{T}^d$, from which the Dirichlet case follows by a symmetrization argument. 
    \item $(\M, g)=(\mathbb{D}, \text{Euclid})$ is the Euclidean closed disk in $\R^2$ and $\omega\cap \partial \M\neq\emptyset$. See Anantharaman-L\'eautaud-Maci\`a \cite[Theorem 1.2]{ALM16}.
\end{enumerate}
Thus a rectangular plate can be damped or controlled from an arbitrarily small interior region, and a circular one from any open set touching its rim, which is typically where a device can be attached.
From a purely mathematical point of view, we can mention the following situations:
\begin{enumerate}
    \item[4.] $(\M, g)$ is the Bunimovich stadium and $\omega$ controls geometrically $\M\setminus R$, $R$ being the rectangular part, see  Burq-Zworski \cite[Theorem 9]{BZ:04}.
    \item[5.] $(\M, g)$ is a compact, boundaryless connected Riemannian surface whose flow has the Anosov property, $\omega$ is any nonempty open set. See Dyatlov-Jin-Nonnenmacher \cite[Theorem 5]{DJN:22}.
    \item[6.] $(\M, g)$ is a compact, boundaryless Riemannian manifold of dimension $d$ and constant curvature $\equiv -1$, and the observation $\mc{C}\psi= a\psi$ is made through a smooth function $a$ on $\M$ such that the set
    $\{\rho\in S^*\M\ |\ a^2(\Phi_t(\rho))=0,\ \forall t\in \R\}$
    has Hausdorff dimension $<d$. Here $\Phi_t$ is the bicharacteristic flow on $T^*\M$. See Anantharaman-Rivi\`ere \cite[Theorem 2.5]{AR:12}.
\end{enumerate}

In all these situations, \cref{assumOBS} is far weaker than what is usually required when working directly with the plate equation, where $\omega$ is typically asked to contain a collar neighborhood of the boundary $\partial\M$, irrespective of the boundary conditions; see for instance \cite{Kim92, GW13, ET:15}. This condition is related to multiplier-type methods, and it is known to be stronger than, for instance, the GCC; see Miller \cite{Mil03}. This list is certainly not exhaustive, but it already highlights the variety of geometrical situations covered by \cref{assumOBS}. Regarding observability of linear Schrödinger equations in unbounded domains, we refer to Prouff \cite{Pro25} and the references therein.

\subsection{Main results}
\subsubsection{Unique continuation property} We assume that $f\in C^\infty(\M\times\R)$ is \emph{uniformly} analytic with respect to the first variable, that is, for every compact $K\subset \R$ and $m\in \N$, there exist $C, R>0$ such that for every $k\in\N$
\begin{align}\label{NLP:analytic}
    \max_{s\in K}\norm{\partial_s^k f(\cdot, s)}_{C^m(\M)}\leq CR^kk!.
\end{align}
We further assume that $f$ has $0$ as an equilibrium point and is defocusing\footnote{Although not standard, we decided to borrow the terminology \emph{defocusing} from the dispersive PDE literature, due to the structural similarities with the defocusing wave and Schrödinger equations.}, that is, 
\begin{align}\label{NLP:defdefocusing}
        f(\cdot, 0)=0\quad \text{ and }\quad sf(x, s)\geq 0\   \textnormal{ for all } (x, s)\in\M\times\R.
\end{align}
Our first main result is a new unique continuation property for solutions of \eqref{eq:nlp}, where the shape of the nonlinearity dictates what kind of equilibrium we will obtain.

\begin{theorem}\label{thm:UCP-NLP}
    Let $f$ satisfy \eqref{NLP:analytic} with $f(\cdot, 0)_{|_{\partial\M}}=0$. Let $0<T_0<T$ and $\omega\Subset\widetilde{\omega}$ be two nonempty open sets such that $\omega$ satisfies \cref{assumOBS} in time $T_0$. If a solution $U=(u,\partial_t u)\in C^0([0, T], H_D^2\times L^2(\M))$ of \eqref{eq:nlp} satisfies $\partial_t u=0$ in $[0, T]\times \widetilde{\omega}$, then $\partial_t u=0$ in $[0, T]\times \M$ and $u$ is an equilibrium point of \eqref{eq:nlp}, that is, %solution of 
            \begin{align}\label{thm:eq:nlw-equilibrium}
                \left\{\begin{array}{rl}
                \Delta_g^2 u+\beta u+f(x, u)=0  &\ x\in\textnormal{Int}(\M),\\
                u_{|_{\partial\M}}=\Delta u_{|_{\partial\M}}=0          &\ x\in \partial\M.
                \end{array}\right.
            \end{align}
    If, moreover, the nonlinearity $f$ satisfies the defocusing assumption \eqref{NLP:defdefocusing}, then $u\equiv 0$.
\end{theorem}
\begin{remark}
    Regarding the analyticity hypothesis \eqref{NLP:analytic}, it could be weakened, since we only need it to hold for a sufficiently high regularity index $m$ so that the propagation mechanism described in \cref{sec:ucp} applies. To not overcomplicate the statements and avoid the index bookkeeping, we choose to state it in this form. We refer to \cite[Definition 1.1, Section 3.1]{FLL24} for further details on this definition.
\end{remark}

To our knowledge, most available unique continuation results for linear plates rely on Carleman estimates, requiring strong geometrical assumptions related to pseudoconvexity but low regularity on the coefficients. One of the earliest results is due to Kim \cite{Kim92}, where a global result with application to control is obtained for an observation in a collar neighborhood of the boundary. Later on, some Carleman estimates were obtained by Isakov \cite{Isa97}, relying on the Schrödinger factorization of the plate operator and on the anisotropic Carleman estimates developed in \cite{Isa93}. This result was later extended by Eller and Toundykov \cite{ET:15} to a low regularity setting, under the assumption of an observation in a collar neighborhood of the boundary. We further refer to \cite{LTY03, Wan07, LeRZ23, IM25} and the references therein for works addressing Carleman estimates for plate equations with different boundary conditions and applications. In this direction, Filippas, Laurent, and Léautaud \cite{FLL24} obtained a unique continuation result for the linear plate operator with lower-order terms assuming Gevrey-$2$ regularity in time and allowing fairly general geometries, see \cref{thm:UCP-linear-plate} below. This result has to be compared with the remarkable theorem of Tataru \cite{Tat95,Tat99}, Robbiano-Zuily \cite{RZ:98} and H\"ormander \cite{Hor:97}, which notably does not apply to the plate operator, since it has double characteristics \cite[Appendix B]{FLL24}.

When moving towards nonlinear results of this kind, very few are available, and those mostly rely on results for the linear equation. Here, relying on the analyticity of the nonlinearity, we use a recent result due to Laurent and the author \cite[Theorem 1.7]{LL24}, in which the nonlinear structure of the equation plays a crucial role to propagate analyticity in time from the vanishing zone to the full solution; this in turn allows us to apply the aforementioned unique continuation result for linear plates. Hence, \cref{thm:UCP-NLP} substantially improves upon the existing literature on unique continuation for nonlinear plates with hinged boundary conditions regarding the geometric conditions under which such a property holds.
As will be seen below, a unique continuation property is key to study the controllability and long-time dynamics of nonlinear models. In particular, it is one of the main obstructions regarding the geometries that are allowed on the control and damping zones, see for instance \cite{GW13, JL13}.

\subsubsection{Semiglobal stabilization}\label{NLP:s:introstab} Let us consider the damped nonlinear plate
\begin{align}\label{eq:nlp-stab}
    \left\{\begin{array}{rl}
        \partial_t^2 u+\Delta_g^2 u+\beta u+\gamma^2(x)\partial_t u+f(x, u)=0  &\ (t, x)\in [0, +\infty)\times \Int(\M),\\
        u_{|_{\partial\M}}=\Delta u_{|_{\partial\M}}=0            &\ (t, x)\in [0, +\infty)\times\partial\M,\\
        (u, \partial_t u)(0)=(u_0, u_1) &\ x\in \M,
    \end{array}\right.
\end{align}
where $\gamma\in C^\infty(\M)$ satisfies $\gamma(x)\geq \gamma_0>0$ for $x\in\omega$ and the nonlinearity $f$ satisfies the defocusing assumption \eqref{NLP:defdefocusing}. The nonlinear energy $E\in C^0(H_D^2\times L^2(\M), \R^+)$ associated to $U=(u, \partial_t u)$ solution of \eqref{eq:nlp-stab} is given by
\begin{align*}
    E(U):=\dfrac{1}{2}\int_\M \big(|\partial_t u|^2+|\Delta u|^2+\beta |u|^2\big)dx+\int_\M V(x, u)dx,
\end{align*}
where $V(x, u)=\int_0^u f(x, s)ds$. Since $d\leq 3$, due to Sobolev embedding the energy $E$ is well-defined. Moreover, $E$ is non-increasing along the trajectories: formally,
\begin{align}\label{eq:energyidentity}
    \partial_t E(U)=-\int_\M \gamma^2(x)|\partial_t u|^2 dx\leq 0.
\end{align}
We are interested in whether the damping term $+\gamma^2\partial_t u$ stabilizes any solution of \eqref{eq:nlp-stab} to $0$ at an exponential rate. Our second main result is the following.

\begin{theorem}\label{thm:NLP-decay}
    Assume that $\omega$ satisfies \cref{assumOBS}. If $f$ satisfy \eqref{NLP:analytic}-\eqref{NLP:defdefocusing}, then for any $E_0\geq 0$, there exist $C>0$ and $\kappa>0$ such that, for any solution $u$ of \eqref{eq:nlp-stab} with $E(U(0))\leq E_0$, it holds 
    \begin{align}\label{ineq:ED}\tag{ED}
    E(U(t))\leq Ce^{-\kappa t}E(U(0)).
    \end{align}
\end{theorem}

It is well known that \eqref{ineq:ED} is equivalent to an observability inequality for the nonlinear equation. We will prove the latter by a compactness-uniqueness method, which can be roughly summarized as follows:
\begin{center}
\emph{linear observability} + \emph{unique continuation} $\Longrightarrow$ \emph{nonlinear observability}.
\end{center}
This scheme was introduced, although not explicitly, by Dehman, Lebeau and Zuazua \cite{DLZ03} to stabilize the semilinear defocusing wave equation in $\R^3$ with damping active outside a ball, and has since been employed for wave and Schrödinger equations in different geometric settings \cite{DGL06, Lau10:3d, JL13, JL20, Per23}.
Coming to our case, the high-frequency part of the proof combines the linear observability inequality with compactness properties of the nonlinearity inherited from the wellposedness framework. This replaces the propagation argument usually handled by microlocal analysis techniques, for instance, under the GCC in the context of the wave and Schrödinger equations. We remark that the propagation part is already contained in \cref{assumOBS}. The low-frequency part of the argument is handled by the unique continuation property \cref{thm:UCP-NLP}, which is available due to the analyticity assumption on the nonlinearity. In this spirit, new observability results for the Schrödinger equation in bounded domains will most likely widen the validity of our results.

A drawback of the compactness-uniqueness method is that it is nonconstructive: both $C$ and $\kappa$ depend on $E_0$. However, since solutions are small for large times, one may hope for a decay rate close to that of the linear equation; see Laurent-Rosier-Zhang \cite{LRZ10} for the KdV equation. Whether the dependence of $C$ on $R_0$ can be dropped is an open problem. We refer to Le Balc'h-Martin \cite{LBM23:stab-NLS} for the uniform global stabilization of the subcritical Schrödinger equation on the torus by a direct method, using nonlinear Carleman estimates and Morawetz-type estimates.

\subsubsection{Controllability and dynamics}\label{I:s:control} Let us consider the controlled system
\begin{align}\label{eq:nlp-control}
    \left\{\begin{array}{rl}
        \partial_t^2 u+\Delta_g^2 u+\beta u+f(x, u)=\mathbbm{1}_{\widetilde{\omega}} h  &\ (t, x)\in [0, T]\times \Int(\M),\\
        u_{|_{\partial\M}}=\Delta u_{|_{\partial\M}}=0            &\ (t, x)\in [0, T]\times\partial\M,\\
        (u, \partial_t u)(0)=(u_0, u_1) &\ x\in \M,
    \end{array}\right.
\end{align}
with $h\in L^\infty([0, T], L^2(\M))$ and $\widetilde{\omega}\subset\M$ a nonempty open set. We relax \eqref{NLP:defdefocusing} by an asymptotic defocusing assumption: we assume that there exists $R>0$ such that
\begin{align}\label{cond:asymptf}
    f(\cdot, 0)_{|_{\partial\M}}=0\ \text{ and }\ sf(x, s)\geq 0\ \text{ for every }\ (x, s)\in\M\times\{|s|\geq R\}.
\end{align} 
The controllability result is the following.

\begin{theorem}\label{thm:NLP-control}
    Let $f$ satisfy \eqref{NLP:analytic}-\eqref{cond:asymptf}. Let $\omega\Subset\widetilde{\omega}$ be two nonempty open sets such that $\omega$ satisfies \cref{assumOBS} in time $T_0>0$. For any bounded set $\mathbb{B}\subset H_D^2\times L^2(\M)$ there exists $T_{\mathbb{B}}>0$ such that for any $(u_0, u_1),~(\widetilde{u}_0, \widetilde{u}_1)\in \mathbb{B}$, there exist $T\in (0, T_{\mathbb{B}})$ and a control $h\in L^\infty([0, T], L^2(\M))$ supported in $(0, T)\times \widetilde{\omega}$ such that the unique solution $(u, \partial_t u)\in C^0([0, T], H_D^2\times L^2(\M))$ of \eqref{eq:nlp-control} satisfies $(u(T), \partial_t u(T))=(\widetilde{u}_0, \widetilde{u}_1)$.
\end{theorem}

\begin{remark}
    Condition \eqref{cond:asymptf} allows foundations that soften at moderate amplitude and harden at large amplitude, for instance $f(x, u)=k(x)|u|^2u+\ell(x)|u|^4u$ with $\ell>0$ and $k$ with no fixed sign.% This is physically meaningful as described in \cref{ex:NL}. 
\end{remark}

Under the asymptotically defocusing assumption \eqref{cond:asymptf}, the equation may admit non-trivial equilibria, so the usual scheme for the defocusing case no longer applies.
We follow instead a control strategy introduced by Joly and Laurent \cite{JL14} for the semilinear wave equation with non-trivial dynamics. It proceeds in two steps. First, we show that the dynamics of the damped plate \eqref{eq:nlp-stab} is gradient and admits a compact connected global attractor \cref{thm:attractorA}; this rests on the unique continuation property \cref{thm:UCP-NLP} and the decay of the linear damped semigroup. Second, we travel within the attractor: reaching it through the damped dynamics, performing perturbative local control around equilibria, which hinges on the linear observability with potential \cref{thm:NLlocalcontrol}, and use time reversibility to steer the solution to any target. The notions from dynamical systems used here are gathered in \cref{appendix:DS}; see Raugel \cite{Rau02}.

\begin{remark}
    The results presented so far are most likely to hold true in higher dimensions $d\geq 4$ under additional growth assumptions on $f$ and its derivative. For $d=4$, where the embedding $H_D^2\hookrightarrow L^\infty$ is no longer available, polynomial growth should allow one to control the nonlinear energy by the linear one and to keep compactness properties similar to \cref{prop:regnonlinearity} below. For $d\geq 5$, one may require Strichartz estimates, which to our knowledge are not available for the plate equation in the geometries we are concerned with; we mention  Pausader \cite{Pau07} for Strichartz estimates in $\R^d$.
\end{remark}

\subsubsection{Brief literature and further comments} Given the extensive literature on the topic, we mainly focus on results concerning nonlinear models.

\paragraph{\textit{Control, stabilization and long-time dynamics of nonlinear plates}} 

Few results in the literature address the global exact controllability of nonlinear plate equations. Lasiecka and Triggiani \cite{LT91} considered globally Lipschitz nonlinearities with a control acting on the entire boundary, relying on linear results and a suitable global inversion theorem; Zhang \cite{Zha01} studied a nonlinearity with superlinear growth, where the analysis relied on deriving a suitable observability estimate for a plate equation with potential via Carleman estimates. Unless stated otherwise, the following results all hold in dimension $d=2$ with flat metric, with damping/control zone equal to a full collar neighborhood of the boundary. In \cite{ET:15}, the authors studied uniform stabilization and control properties for several nonlinear plate models with a variety of boundary conditions, notably with strictly dissipative polynomial nonlinearity and the Berger plate model, the control part relying partly on stabilization through a suitable feedback. Local control for Berger plates was obtained by Cîndea and Tucsnak \cite{NM09}, under the assumption that the linearized system is controllable. Geometrically constrained nonlinear damping is addressed by Geredeli, Lasiecka and Webster \cite{GLW13} and Geredeli and Webster \cite{GW13}, where unique continuation for such models is emphasized as a challenging problem. Regarding long-time dynamics, the monograph of Chueshov and Lasiecka \cite{CL08} treats nonlinear plates in an abstract setting covering local nonlinearities (including our model) and the von Kármán and Berger models, under a variety of boundary conditions; the von Kármán case is developed thoroughly in \cite{CL10}.
The analyticity of the nonlinearity also governs the asymptotic behavior: under damping active everywhere, by Łojasiewicz-Simon arguments \cite{HJ01}, every trajectory of the damped equation converges to a single equilibrium; convergences rates are established for von Kármán models in \cite{Chu12}.

\paragraph{\textit{Toward von Kármán and Berger models}}
Studying models like \eqref{eq:nlp} may open a path toward the more physically relevant Berger and von Kármán models, whose long-time dynamics also feature non-trivial equilibria. The main obstruction is a lack of regularity of the nonlinearity: the von Kárman nonlinearity $f(w)=[w, \Phi(w, w)]$, where $[\cdot,\cdot]$ is the Von Kármán bracket and $\Phi(v, w)$ is the Airy stress function, is Lipschitz on bounded sets from $H_D^2$ to $L^2$, with no gain of regularity above the energy level \cite[Corollary 1.4.5]{CL10}. It thus lacks the compactness in the energy space required to apply \cite[Theorem 1.7]{LL24} and propagate analyticity in time from the observed region to the full solution. This is in sharp contrast with our setting, where the gain of regularity on the nonlinearity (see \cref{prop:regnonlinearity}) is precisely what lets us use the observability of the linear equation as a black box for the propagation argument. A refinement of the propagation of analyticity mechanism \cite{Loy25:NLP} suggests that a first step would be to understand the microlocal propagation of compactness and regularity for a linearized operator $\partial_t^2+\Delta_g^2+P_{\mathrm{VK}}$, where $P_{\mathrm{VK}}$ must account for the nonlocality of the von Kármán nonlinearity. Once these propagation mechanisms are understood, we believe the control strategy of \cref{I:s:control}, combined with suitable local control results, would offer a natural route to large data control for such models. In this direction, we mention Bournissou, Ervedoza and Tucsnak \cite{TBE24}, where local exact controllability around analytic equilibria is proved for the von Kármán model with hinged boundary conditions under the GCC.

\paragraph{\textit{On the boundary conditions}}
Our analysis rests on the hinged conditions through the Schrödinger factorization of the plate operator, which has no direct analogue for the clamped conditions $u=\partial_\nu u=0$ or the free conditions. Transferring our results to those settings is an interesting open problem: already at the linear level, neither the observability nor the unique continuation theory is as well understood as in the hinged case. For the linear damped plate under general boundary conditions satisfying the Lopatinski\u{\i}-\v{S}apiro condition, Le Rousseau and Zongo \cite{LeRZ23} proved, by microlocal and resolvent methods and without any geometric assumption on the damping, that the decay rate is at least logarithmic, which is in sharp contrast with the exponential decay obtained under \cref{assumOBS}. We are not aware of any result for nonlinear plates under such general boundary conditions. Regarding the stationary case, we refer to Le Rousseau and Robbiano \cite{LeRR20} where spectral inequalities and resolvent estimates are obtained for the bi-Laplace operator with clamped boundary conditions.

\subsection{Outline of the article} \Cref{sec:pre} provides preliminaries about the plate equation. \Cref{sec:ucp} contains the proof of the unique continuation result \Cref{thm:UCP-NLP}. \Cref{sec:dyn} is devoted to the damped plate equation where we prove \cref{thm:NLP-decay}. \Cref{sec:control} is devoted to the controllability result and contains the proof of \cref{thm:NLP-control}. \cref{appendix:NL} contains some nonlinear estimates and \cref{appendix:DS} gathers some dynamical systems results.

\subsection*{Acknowledgments}
I would like to thank Camille Laurent for drawing my attention to this problem, as well as for the helpful discussions, encouragement, and patient guidance. I would like to thank the reviewers for their thorough comments, which have significantly improved the quality and readability of the article. 
This project has received funding from the European Union's Horizon 2020 research and innovation programme under the Marie Sk\l{}odowska-Curie grant agreement No 945332.
% This article was completed while I was affiliated with Laboratoire Jacques-Louis Lions, Sorbonne Université, France, and I would like to acknowledge this institution for providing an exceptional mathematical environment and excellent working conditions. 

\section{Preliminaries}\label{sec:pre}

\subsection{Functional setting and well-posedness}\label{sec:s-wellp} Let $A_0=-\Delta_g$ be the Laplace-Beltrami operator equipped with Dirichlet boundary conditions if $\partial\M\neq\emptyset$. Recall that $A_0: D(A_0)\to L^2(\M)$ is self-adjoint and nonnegative, with compact resolvent and domain $D(A_0)=H^2(\M)\cap H_0^1(\M)$ on $L^2(\M)$. For $s\geq 0$, we introduce the scale of Sobolev spaces $H_D^s=D((1+A_0)^{s/2})$ equipped with the inner product
\begin{align*}
    \inn{\phi, \psi}_{H_D^s}=\inn{(1+A_0)^{s/2}\phi, (1+A_0)^{s/2}\psi}_{L^2(\M)},
\end{align*}
see \cite{G:67}, and define $H_D^{-s}=(H_D^{s})'$ by duality with $L^2(\M)$ as pivot space, with duality pairing extending the $L^2$-inner product. For all $s, \nu\in \R$, the operator $(1+A_0)^{s/2}$ extends to an isomorphism $H_D^{s+\nu}\to H_D^{\nu}$, and $H_D^{s}\hookrightarrow H_D^{s'}$ with compact embedding whenever $s>s'$. Under our assumptions on $\beta$, the operator $A_0^2+\beta$ is self-adjoint and strictly positive with $D((A_0^2+\beta)^{\sigma/2})=H_D^{2\sigma}$ and equivalent norms. Set $X=H_D^2\times L^2(\M)$ and, for $\sigma\in \R$,
\begin{align}\label{defSob}
    X^\sigma&:=D((A_0^2+\beta)^{(1+\sigma)/2})\times D((A_0^2+\beta)^{\sigma/2})=H_D^{2+2\sigma}\times H_D^{2\sigma},
\end{align}
equipped with energy inner product
\begin{align*}
    \inn{(\phi_0, \phi_1), (\psi_0, \psi_1)}_{X^\sigma}
    =\inn{(A_0^2+\beta)^{1/2}\phi_0, (A_0^2+\beta)^{1/2}\psi_0}_{H_D^{2\sigma}}
    +\inn{\phi_1, \psi_1}_{H_D^{2\sigma}},
\end{align*}
whose induced norm is equivalent to the product one. We identify $X^{-1}=L^2(\M)\times H_D^{-2}=X'$ with the dual $X'$ of $X$ through the duality pairing
\begin{align*}
    \inn{(\phi_0, \phi_1), (\psi_0, \psi_1)}_{X'\times X}=\inn{\phi_1, \psi_0}_{H_D^{-2}\times H_D^2}-\int_\M \phi_0\psi_1dx,
\end{align*}
where $\inn{\cdot, \cdot}_{H_D^{-2}, H_D^2}$ denotes the duality pairing between $H_D^{-2}$ and $H_D^2$.

Introduce the densely defined operator $\bA: D(\bA)\to X$ given by
\begin{align*}
    \bA=\left(\begin{array}{cc}
        0 & I \\
        -A_0^2-\beta I & 0
        \end{array}\right)\ \text{ with }\ D(\bA)=X^1.
\end{align*}
By Stone's theorem, $\bA$ generates a unitary $C_0$-group $t\mapsto e^{t\bA}$ on $X$, which can be further extended to $X^\sigma$ for $\sigma\geq 0$, as $\bA$ commutes with the diagonal action $(1+A_0)^\sigma:(\psi_0, \psi_1)\mapsto((1+A_0)^\sigma\psi_0,
(1+A_0)^\sigma\psi_1)$, which is an isomorphism from $X^\sigma$ into $X$. By duality this extends to $\sigma<0$ as well. We set $F\in C^0(X)$ to be the map
\begin{align*}
    F: (u, v)\in X\longmapsto (0, -f(\cdot, u))\in X.
\end{align*}
First, we show that $F$ enjoys a regularity gain when defined on the energy space $X$.

\begin{proposition}\label{prop:regnonlinearity}
    Assume that $f(\cdot, 0)_{|_{\partial\M}}=0$. The nonlinearity $F$ maps $X$ into $X^1$ and, for every bounded set $\mathbb{B}\subset X$, there exists $C(\mathbb{B})>0$ such that, for every $\sigma\in[0, 1]$ and any $U$, $V\in \mathbb{B}$,
    \begin{align}\label{ineq:FlipXsigma}
        \norm{F(U)-F(V)}_{X^\sigma}\leq C(\mathbb{B})\norm{U-V}_X.
    \end{align}
    In particular, $F(\mathbb{B})$ is relatively compact in $X^\sigma$ for every $\sigma\in [0, 1)$.
\end{proposition}
\begin{proof}
    Since $d\leq 3$, by Sobolev embedding $H_D^2\hookrightarrow C^0(\M)$ and \cref{app:lem:NLcomposition} (with $s=2$ and $N=3$) applied to to $f-f(\cdot, 0)$, we get $\norm{f(u)}_{H_D^2}\leq\norm{f(\cdot, 0)}_{H^2}+C(\norm{u}_{L^\infty(\M)})\norm{u}_{H_D^2}$. Since $f(\cdot, 0)_{|_{\partial\M}}=0$, then $f: H_D^2\to H_D^2$ is well-defined.
    Moreover, for any $u,~v\in H_D^2$, by writing
    \begin{align*}
        f(\cdot, u)-f(\cdot, v)=(u-v)\int_0^1 \big(f_s'(\cdot, v+\tau(u-v))-f_s'(\cdot, 0)\big)d\tau+(u-v)f_s'(\cdot, 0),
    \end{align*}
    we can apply \cref{app:lem:NLcomposition} to $f_s'-f_s'(\cdot, 0)\in C^3(\M\times\R, \R)$ (with $s=2$ and $N=3$) and use that $H^2$ is an algebra to obtain
    \begin{align*}
        \norm{f(\cdot, u)-f(\cdot, v)}_{H_D^2}\leq C\norm{u-v}_{H_D^2}\big(1+\norm{u}_{H_D^{2}}+\norm{v}_{H_D^2}\big),
    \end{align*}
    where the constant $C$ depends only on $f$ and on the $L^\infty$ norm of both $u$ and $v$. For any $\sigma\in [0, 1]$, by Sobolev embedding $H_D^2\hookrightarrow H_D^{2\sigma}$, for any $U$, $V\in\mathbb{B}$,
    \begin{align*}
        \norm{F(U)-F(V)}_{X^\sigma}&=\norm{f(\cdot, u)-f(\cdot, v)}_{H_D^{2\sigma}}\leq C\norm{u-v}_{H_D^2}\leq C\norm{U-V}_{X},
    \end{align*}
    where $C$ only depends on $f$ and $\mathbb{B}$. In particular, for $\sigma\in [0, 1)$ the embedding $H_D^2\hookrightarrow H_D^{2\sigma}$ is compact by Rellich-Kondrachov and thus the set $\{f(u)\ |\ (u, v)\in \mathbb{B}\}$ is relatively compact in $H_D^{2\sigma}$. Hence, as $F(u, v)=(0, -f(\cdot, u))$, we see that $F$ maps bounded sets of $X$ into relatively compact sets of $X^\sigma$.
\end{proof}

To treat \eqref{eq:nlp-stab} and \eqref{eq:nlp-control}, for $a\in L^\infty(\M)$ and a source term $p$, let us consider
\begin{align}\label{eq:nlp-source}
    \left\{\begin{array}{rl}
        \partial_t^2 u+\Delta_g^2 u+\beta u+a(x)\partial_t u+f(x, u)=p   &\ (t, x)\in [0, T]\times \Int(\M),\\
        u_{|_{\partial\M}}=\Delta u_{|_{\partial\M}}=0      &\ (t, x)\in [0, T]\times\partial\M,\\
        (u, \partial_t u)(0)=(u_0, u_1) &\ x\in \M,
    \end{array}\right.
    \end{align}
If $U_0=(u_0, u_1)\in X$ and $U=(u, \partial_t u)$, we can recast it as a first order system
\begin{align*}
    \left\{\begin{array}{cl}
        \partial_t U=(\bA+\bB_a)U+F(U)+\bP,&\ t>0,\\
        U(0)=U_0,&
    \end{array}\right. \quad \bB_a=\begin{pmatrix}
        0  & 0 \\
        0  & -a
    \end{pmatrix},\quad \bP=\begin{pmatrix} 0 \\ p\end{pmatrix}.
\end{align*}
We now state the global well-posedness of \eqref{eq:nlp-source}.

\begin{proposition}\label{prop:wellposed}
    Let $f$ be as above satisfying \eqref{cond:asymptf} and let $a\in L^\infty(\M)$ with $a\geq 0$. For any $T>0$, any $U_0=(u_0, u_1)\in X$ and $p\in L^1([0, T], L^2(\M))$ there exists a unique solution $U=(u, \partial_t u)\in C^0([0, T], X)$ of \eqref{eq:nlp-source}.
    
    For $0\leq s\leq t\leq T$ it satisfies
    \begin{align}\label{eq:energy-1}
        E(U(t))-E(U(s))=\int_s^t\int_\M (p-a\partial_t u)\partial_t udxd\tau,
    \end{align}
    and the flow map $(U_0, p)\in X\times L^1([0, T], L^2(\M))\longmapsto U\in C^0([0, T], X)$ is Lipschitz on bounded sets.
    If $p\equiv 0$, the energy is non-increasing on $[0, +\infty)$, the solution extends to a unique $U\in C^0(\R, X)$ and for any $T>0$ the flow map $U_0\in X\longmapsto U\in C^0([-T, T], X)$ is Lipschitz on bounded sets.
\end{proposition}
\begin{proof}
    By \cref{prop:regnonlinearity} the map $F: X\to X$ is Lipschitz on bounded sets, and since $\bB_a\in \mc{L}(X)$, so is $\bB_a+F: X\to X$. Since $\bA$ generates a unitary $C_0$-group, for $U_0\in X$ and $p\in L^1([0, T], L^2(\M))$, by classical Picard iteration \cite[Theorem 6.1.4]{Pazy82} we obtain a unique mild solution $U\in C^0([0, T_{\mathrm{max}}), X)$ with the blow-up alternative: either $T_{\mathrm{max}}=+\infty$ or $\limsup_{t\to T_{\mathrm{max}}}\norm{U(t)}_X=+\infty$.

    Let us pick $T'<T_{\mathrm{max}}$. By a first order expansion $f(\cdot, u+\delta)-f(\cdot, u)-f'(\cdot, u)\delta=\int_0^1 [f_s'(\cdot, u+\tau\delta)-f_s'(\cdot, u)]\delta d\tau$.
    Using the same kind of bound as \cref{prop:regnonlinearity} and uniform continuity of $f_s'$ on compact sets, it follows that $F\in C^1(X, X)$. Hence, for $U_0\in D(\bA)$ and $p\in C^{1}([0, T'], L^2(\M))$, the mild solution is actually a strong solution $U\in C^0([0, T'], D(\bA))\cap C^1([0, T'], X)$, see \cite[Theorem 6.1.5]{Pazy82}. For such strong solutions, for $s\leq t$, the energy $E$ satisfies the identity \eqref{eq:energy-1}. Also, since $F$ is Lipschitz on bounded sets, by Duhamel's formula and a Grönwall-type estimate, the map $(U_0, p)\mapsto U$ is Lipschitz on bounded sets of $X\times L^1([0, T], L^2(\M))$.

    To globalize the solution we need to obtain some suitable bounds on the energy. First, given that $\M$ is compact and assumption \eqref{cond:asymptf} forces that $s\mapsto V(\cdot, s)$ is non-increasing on $(-\infty, -R)$ and non-decreasing on $(R, +\infty)$, it follows that $V$ is bounded below on $\M\times\R$, that is, $C_0:=-\inf_{(x, s)\in \M\times\R} V(x, s)<+\infty$. It follows that
    \begin{align}\label{ineq:energybelow}
        E(U)\geq \dfrac{1}{2}\norm{U}_X^2-C_1,\quad C_1:=\meas(\M)C_0.
    \end{align}
    By \eqref{eq:energy-1}, $a\geq 0$, Cauchy-Schwarz in space and Hölder in time, followed by \eqref{ineq:energybelow},
    \begin{align}\label{ineq:energyabove}
        E(U(t))\leq E(U(s))+\sqrt{2}\norm{p}_{L^1([s, t], L^2(\M))}\sup_{\tau\in[s, t]}\big(E(U(\tau))+C_1\big)^{1/2}.
    \end{align}
    If we set $\Psi(t):=\sup_{\tau\in [0, t]}\big(E(U(\tau))+C_1\big)^{1/2}$, then for $s=0$ and $t=T'$, estimate \eqref{ineq:energyabove} above becomes quadratic in $\Psi(T')$, from which we obtain $\Psi(T')\leq \sqrt{2}\norm{p}_{L^1([0, T], L^2(\M))}+\Psi(0)$.
    By the energy bound \eqref{ineq:energybelow}, we have
    \begin{align*}
        \sup_{\tau\in [0, T']} \norm{U(\tau)}_X^2\leq 2\sup_{\tau\in [0, T']}(E(U(\tau))+C_1)\leq C\Big(\norm{p}_{L^1([0, T], L^2(\M))}^2+E(U_0)+C_1\Big).
    \end{align*}
    This bound does not depend on $T'$, hence $T_{\mathrm{max}}=+\infty$ by the blow-up criterion.
    
    Using that $(U_0, p)\mapsto U$ is Lipschitz on bounded sets of $X\times L^1([0, T], L^2(\M))$, that $E$ is continuous on $X$ and a density argument, \eqref{eq:energy-1} remains valid for all $(U_0, p)\in X\times L^1([0, T], L^2(\M))$.
    
    Let $p\equiv 0$ and put $W(s):=(u(-s), -\partial_t u(-s))$, which solves \eqref{eq:nlp-source} with $-a$ instead of $a$. Since the local theory does not depend on the sign of $a$, by \eqref{eq:energy-1} we get
    \begin{align*}
        E(W(s))+C_1\leq E(W(0))+C_1+2\norm{a}_{L^\infty(\M)}\int_0^s (E(W(\tau))+C_1)d\tau,\ s\geq 0,
    \end{align*}
    which yields $E(U(t))+C_1\leq (E(U_0)+C_1)e^{2\norm{a}_{L^\infty(\M)}|t|}$ for any $t\leq 0$. By the blow-up criterion, $U$ extends to a unique solution in $C^0(\R, X)$, satisfying \eqref{eq:energy-1} and whose flow map $U_0\in X\to C^0([-T, T], X)$ is Lipschitz on bounded sets.
\end{proof}

\begin{remark}
    Note that the condition $\Delta u_{|_{\partial\M}}=0$ which does not make sense at this level of regularity is encoded in the functional setting through $D(\bA)$. The nonlinear equation \eqref{eq:nlp-source} is meant in the sense of Duhamel's formula.
\end{remark}

\subsection{Linear observability}\label{s:obs} The starting point is the observability of the free hinged plate equation
\begin{align}\label{eq:freeplate}
    \left\{\begin{array}{rl}
        \partial_t^2 z+\Delta_g^2 z=0  &\ (t, x)\in [0, T]\times \Int(\M),\\
        z_{|_{\partial\M}}=\Delta_g z_{|_{\partial\M}}=0            &\ (t, x)\in [0, T]\times\partial\M,\\
        (z, \partial_t z)(0)=(z_0, z_1) &\ x\in \M,
    \end{array}\right.
\end{align}
borrowed from \cite[Section 4]{LL24}, where Lebeau's strategy to transfer observability for the linear Schrödinger to the plate equation is described in an abstract setting, thus explaining \cref{assumOBS}.

\begin{proposition}\label{prop:obs-plate}\cite[Proposition 4.10]{LL24}
 Let $\omega$ satisfy \cref{assumOBS} in time $T_0>0$. Then, for any $T_0<T<+\infty$ and any $b_{\omega}\in C^{\infty}(\M)$ so that $b_{\omega}=1$ on $\omega$, there exists $C>0$ such that for any $s\in [0,2]$, any $(z_0, z_1)\in H_D^{2+s}\times H_D^s(\M)$ and associated solution $z$ of \eqref{eq:freeplate}, we have
       \begin{align}
       \label{thm:ineq:obs-plate-1}
        \norm{(z_0, z_1)}_{H^{2+s}_D\times H^s_D(\M)}^2\leq C\int_0^T \norm{b_\omega \Delta_g z(t)}_{H^s_D(\M)}^2 dt,\quad \text{ if } \partial \M\neq \emptyset,\\
    \label{thm:ineq:obs-plate-1bis}
        \norm{(z_0, z_1)}_{H^{2+s}\times H^s(\M)}^2\leq C\int_0^T \norm{b_\omega z(t)}_{H^{2+s}(\M)}^2 dt, \quad \text{ if } \partial \M= \emptyset.
    \end{align}
\end{proposition}

In the above result, up to an equivalence due to the boundary conditions, we observe the first coordinate through the observation operator $\bC_0(\psi_0, \psi_1)=(b_\omega \psi_0, 0)$ in $\mc{L}(X)$. Below, we argue that observability still holds for the linear plate if we observe the second coordinate through $\bC(\psi_0, \psi_1)=(0, \gamma\psi_1)$ in $\mc{L}(X)$, and that it remains true under after adding a suitable potential term. To handle this, we will use a slight variation of an abstract result \cite[Corollary 1.3, Remark 1.5]{DO18}, which states that the controllability of a compactly perturbed system follows from the exact controllability of the free system and a suitable unique continuation property. By the classical Hilbert Uniqueness Method (HUM) introduced by Lions \cite{Lio88}, the exact controllability of a system on a Hilbert space is equivalent to the observability of its adjoint. In this direction, we can state the following variant of their abstract result.

\begin{lemma}\label{lem:obs-perturbed}
    Let $\A: D(\A)\subset\mc{H}\to \mc{H}$ generate a $C_0$-group on a Hilbert space $\mc{H}$, $\mc{C}\in \mc{L}(\mc{H}, \mc{U})$ for some Hilbert space $\mc{U}$, and let $\mc{K}\in \mc{L}(\mc{H})$ be compact. Set $\A_{\mc{K}}:=\A+\mc{K}$ with $D(\A_{\mc{K}})=D(\A)$. Assume:
    \begin{enumerate}
        \item there exist $T',~C_0>0$ such that $\norm{Z_0}_{\mc{H}}^2\leq C_0\int_0^{T'}\norm{\mc{C}e^{t\A}Z_0}_{\mc{U}}^2dt$ for all $Z_0\in \mc{H}$;
        \item for every $\ld\in\mathbb{C}$, $\A_{\mc{K}}Z=\ld Z,~\mc{C}Z=0$ with $Z\in D(\A_{\mc{K}})$, implies $Z=0$.
    \end{enumerate}
    Then for every $T>T'$ there exists $C>0$ such that $\norm{Z_0}_{\mc{H}}^2\leq C\int_0^{T}\norm{\mc{C}e^{t\A_{\mc{K}}}Z_0}_{\mc{U}}^2dt$ for all $Z_0\in \mc{H}$.
\end{lemma}

\begin{remark}
    In \cref{sec:control} a brief description of the HUM method is given. Since here we are interested in the observability, we do not go further in this direction.
\end{remark}

Below we state the desired observability inequality with a suitable potential.

\begin{corollary}\label{cor:obs-plate}
    Let $V\in C^\infty(\M)$ be nonnegative $V\geq 0$, and let $\omega$ satisfy \cref{assumOBS} in time $T_0>0$. Let $\gamma\in C^\infty(\M)$ be as in \cref{NLP:s:introstab}. For any $T>T_0$ there exists $C>0$ such that, for any $(z_0, z_1)\in H_D^{2}\times L^2(\M)$, we have
    \begin{align}\label{cor:ineq:obs-plate-1}
        \norm{(z_0, z_1)}_{H_D^2\times L^2(\M)}^2\leq C\int_0^T \norm{\gamma \partial_t z(t)}_{L^2(\M)}^2 dt,
    \end{align}
    with $z$ solution of
    \begin{align}\label{eq:obspotential}
    \left\{\begin{array}{rl}
            \partial_t^2 z+\Delta_g^2 z+\beta z+V(x)z=0  &\ (t, x)\in [0, T]\times \Int(\M),\\
            z_{|_{\partial\M}}=\Delta z_{|_{\partial\M}}=0            &\ (t, x)\in [0, T]\times\partial\M,\\
            (z, \partial_t z)(0)=(z_0, z_1) &\ x\in \M,
        \end{array}\right.
    \end{align}
\end{corollary}
\begin{proof}
    Pick $T>T_0$. We will verify the assumptions of \cref{lem:obs-perturbed}.
    If $Z:=(z, \partial_t z)$, it belongs to $C^0([0, T], X)$ and \eqref{eq:obspotential} can then be written as a first order system
    \begin{align*}
        \left\{\begin{array}{cc}
            \partial_t Z=\bA_{\boldsymbol{V}}Z,  & t\in (0, T),  \\
            Z(0)=Z_0,     & 
        \end{array}\right.\quad \bA_{\boldsymbol{V}}:=\bA+\boldsymbol{V},\quad \boldsymbol{V}:=\begin{pmatrix}
            0 & 0 \\ -V(x) & 0
        \end{pmatrix}.
    \end{align*}
    on $X$, with observation operator $\bC\in \mc{L}(X)$ defined as $\bC(\psi_0, \psi_1):=(0, \gamma(x)\psi_1)$. First, note that $\bA_{\boldsymbol{V}}:=\bA+\boldsymbol{V}$ satisfies $D(\bA_{\boldsymbol{V}})=D(\bA)$ and that $\boldsymbol{V}\in \mc{L}(X)$ is a compact operator by the compact embedding $H_D^2\hookrightarrow L^2(\M)$.
    
    \paragraph{\textit{Step 1. Free observability}} We sketch the proof for the more involved case $\partial\M\neq\emptyset$. Let $T'\in (T_0, T)$ and set $\Psi=(\psi, \partial_t\psi)\in C([0, T'], D(\bA))$ to be the solution of $\partial_t\Psi=\bA\Psi$ with $\Psi(0)=(\psi_0, \psi_1)\in D(\bA)$; this justifies the computations below and the resulting estimate extends to $X$ by a standard density argument.
    
    Let $\widetilde{T}\in (T_0, T')$ and $\rho\in C_c^\infty(0, T')$ nonnegative with $\rho=1$ on $(\delta, \widetilde{T}+\delta)$ with $0<\delta<\tfrac{1}{2}(T'-\widetilde{T})$ and $\supp\rho\subset (0, T')$. Taking $b_\omega\in C^\infty(\M)$ with $b_\omega=1$ on $\omega$ and $\supp b_{\omega}\subset \{\gamma(x)>\gamma_0/2\}$, invoking the observability estimate \eqref{thm:ineq:obs-plate-1} for $s=0$ and using that $t\mapsto e^{t\bA}$ is unitary, we get
    \begin{align}\label{ineq:obs-changeC-1}
        \norm{(\psi_0, \psi_1)}_{H_D^2\times L^2}^2\leq C'\int_{\delta}^{\widetilde{T}+\delta} \norm{b_\omega \Delta_g \psi(t)}_{L^2}^2dt\leq C\int_0^{T'} \rho(t)\norm{\gamma \Delta_g \psi(t)}_{L^2}^2dt,
    \end{align}
    as the addition of the term $+\beta \psi$ in the equation does not spoil the observability.
    Multiplying the equation by $\gamma^2\rho\psi$, integration by parts on $[0, T']\times\M$ leads us to
    \begin{multline}\label{eq:obs-changeC-2}
        \int_0^{T'} \rho\norm{\gamma \Delta \psi}_{L^2(\M)}^2dt=\int_0^{T'} \rho\norm{\gamma \partial_t \psi}_{L^2(\M)}^2dt-\beta\int_0^{T'} \rho\norm{\gamma \psi}_{L^2(\M)}^2dt\\+\iint_{[0, T']\times \M} \rho'\gamma^2\psi\partial_t \psi dxdt+\iint_{[0, T']\times\M}\rho[\gamma^2, \Delta]\psi\Delta \psi dxdt.
    \end{multline}
    The first term on the right-hand side is the observation we are looking for. The second term is nonpositive. For the third and fourth terms, by Cauchy-Schwarz, energy conservation in $X=H_D^2\times L^2(\M)$ and $X^{-1/2}=H_D^{1}\times H_D^{-1}$, for $\veps>0$ we have
    \begin{align}\label{ineq:obs-changeC-3}
        \left|\iint_{[0, T']\times \M} \rho'\gamma^2\psi\partial_t \psi dxdt\right|&\leq \veps \norm{(\psi_0, \psi_1)}_{X}^2+C_{\veps}'\int_0^{T'}\norm{\gamma\partial_t \psi(t)}_{L^2(\M)}^2dt,\\
        \label{ineq:obs-changeC-4}\left|\iint_{[0, T']\times\M}\rho[\gamma^2, \Delta]\psi\Delta \psi dxdt\right|&\leq \veps\norm{(\psi_0, \psi_1)}_{X}^2+C_{\veps}\norm{(\psi_0, \psi_1)}_{X^{-1/2}}^2,
    \end{align}
    where we used $[\gamma^2, \Delta]\psi=-2\nabla(\gamma^2)\cdot\nabla \psi-(\Delta\gamma^2)\psi$. Plugging \eqref{eq:obs-changeC-2} into \eqref{ineq:obs-changeC-1}, followed by estimates \eqref{ineq:obs-changeC-3}-\eqref{ineq:obs-changeC-4}, picking $\veps>0$ small enough we obtain
    \begin{align*}
        \norm{(\psi_0, \psi_1)}_{X}^2\leq C\left(\int_0^{T'}\norm{\gamma\partial_t \psi(t)}_{L^2(\M)}^2dt+\norm{(\psi_0, \psi_1)}_{X^{-1/2}}^2\right).
    \end{align*}
    The compact term on the right-hand side is then removed by a classical compactness-uniqueness argument, as done in \cite[Proposition 4.7]{LL24} in an abstract setting.

    \paragraph{\textit{Step 2. Unique continuation}} It remains to check the unique continuation property for eigenfunctions of $\bA_{\boldsymbol{V}}$: if $\bA_{\boldsymbol{V}}(v_\ld, w_\ld)=\ld(v_\ld, w_\ld)$ and $\bC(v_\ld, w_\ld)=0$ with $(v_\ld, w_\ld)\in D(\bA)$ and $\ld\in \mathbb{C}$, then $(v_\ld, w_\ld)=(0, 0)$. For $\ld\neq 0$, the first coordinate of $\bA_{\boldsymbol{V}}(v_\ld, w_\ld)=\ld(v_\ld, w_\ld)$ gives $w_\ld=\ld v_\ld$; since $\gamma w_\ld=0$ this leads to
    \begin{align*}
        \left\{\begin{array}{cl}
            -(\Delta_g^2+\beta+V(x)\big)v_\ld=\ld^2 v_\ld,     & x\in\M, \\
            v_\ld=0,     & x\in\omega.
        \end{array}\right.
    \end{align*}
    In particular, the unique continuation for (evolution) plate operators \cref{thm:UCP-linear-plate} below yields $v_\ld\equiv 0$; since $w_\ld=\ld v_\ld$, it follows that $(v_\ld, w_\ld)=(0, 0)$. For $\ld=0$, we have $w_0=0$ and $            -(\Delta_g^2+\beta+V(x)\big)v_0=0,~x\in\M$.
    Multiplying the latter by $v_0$, we obtain $\norm{\Delta_g v_0}_{L^2(\M)}^2+\beta\norm{v_0}_{L^2(\M)}^2+\int_\M V(x)|v_0|^2dx=0,$
    and using that $V\geq 0$, the boundary conditions and the connectedness of $\M$, we obtain $v_0=0$, hence $(v_0, w_0)=(0, 0)$. 
    \cref{lem:obs-perturbed} then yields \eqref{cor:ineq:obs-plate-1}.
\end{proof}

\begin{remark}
    The hypothesis $V\geq 0$ has only been used to handle the case $\ld=0$.
\end{remark}

\subsection{About the stabilization strategy} We recall the following well-known characterization for exponential decay of \eqref{eq:nlp-stab}, see for instance \cite[Proposition 2.5]{JL13}.

\begin{lemma}\label{lem:ED-equiv}
    Let $f$ satisfy the defocusing assumption \eqref{NLP:defdefocusing}. The property \eqref{ineq:ED} holds if and only if there exists $T>0$ such that for every $E_0$, there exists a constant $C>0$ such that
    \begin{align*}%\label{ineq:ED-equiv}
        E(U(0))\leq C\big(E(U(0))-E(U(T))\big)=C\iint_{[0,T]\times\M} \gamma^2(x)|\partial_t u(t, x)|^2dxdt
    \end{align*}
    for all solutions $u$ of \eqref{eq:nlp-stab} with $E(U(0))\leq E_0$.
\end{lemma}

This link between exponential stabilization and observability inequalities was first introduced by Haraux \cite{Har89} for second order linear systems; see also Trélat \cite[Theorem 5.3.1]{Tre24}. In view of the latter results, we can obtain the observability inequality that leads to the decay of the linear plate with localized damping.

\begin{proposition}\label{prop:lineardecay}
    If $\omega$ satisfies \cref{assumOBS} in time $T_0$, then for any $T>T_0$ there exists $C>0$ such that for any $(z_0, z_1)\in X$, we have
    \begin{align}\label{prop:ineq:obs-dampedlinear}
        \norm{(z_0, z_1)}_{H_D^2\times L^2(\M)}^2\leq C\int_0^T \norm{\gamma \partial_t z(t)}_{L^2(\M)}^2 dt.
    \end{align}
    where $(z, \partial_t z)\in C^0(\R, X)$ is solution of the linear damped plate
    \begin{align}\label{eq:damped-linear}
    \left\{\begin{array}{rl}
            \partial_t^2 z+\Delta_g^2 z+\beta z+\gamma^2(x)\partial_tz=0  &\ (t, x)\in [0, T]\times \Int(\M),\\
            z_{|_{\partial\M}}=\Delta z_{|_{\partial\M}}=0            &\ (t, x)\in [0, T]\times\partial\M,\\
            (z, \partial_t z)(0)=(z_0, z_1) &\ x\in \M.
        \end{array}\right.
    \end{align}
    Moreover, there exist $K>0$ and $\kappa_0>0$ such that for any $\sigma\in[0, 1]$ and $t\geq 0$,
    \begin{align}\label{prop:ineq:damped-decay}
        \norm{e^{t(\bA+\bB_{\gamma^2})}}_{\mc{L}(X^\sigma)}\leq Ke^{-\kappa_0 t}.
    \end{align}
\end{proposition}
\begin{proof}
    The observability estimate \eqref{prop:ineq:obs-dampedlinear} follows from \cref{cor:obs-plate} with $V\equiv 0$ and \cite[Theorem 5.3.1]{Tre24}. Furthermore, by \cite[Theorem 5.3.1]{Tre24}, the observability estimate \eqref{prop:ineq:obs-dampedlinear} yields the exponential decay of the linear semigroup \eqref{prop:ineq:damped-decay} for $\sigma=0$. The case $\sigma=1$ can be obtained from the case $\sigma=0$ and a commutator argument \cite[Section 2.3.2]{HR03}. The full range $\sigma\in [0, 1]$ is then obtained by linear interpolation.
\end{proof}

\section{Unique continuation and equilibria}\label{sec:ucp}
A crucial property on which our analysis relies is a unique continuation property for the nonlinear plate equation \eqref{eq:nlp}. It is built upon two blocks. The first one is a recent result from \cite{FLL24}, stating a unique continuation property for the linear plate equation with time-dependent lower order terms. Actually, this result is stated as a local unique continuation property. However, as pointed out by the authors, such a result can be globalized by propagation through a suitable family of hypersurfaces, whose construction can be found in \cite[Proof of Theorem 6.7]{LL19}. The global version is stated below.

\begin{theorem}{\cite[Section 1.3.5]{FLL24}}\label{thm:UCP-linear-plate}
    Let $T>0$ and $\M = \Int\M\sqcup \partial\M$ be a connected smooth manifold with or without boundary $\partial\M$.
	Suppose that $g \in W^{3,\infty}_{\loc}(\Int(\M))$ is a Riemannian metric on $\Int\M$, that $\mathsf{q} \in \mc{G}^2((0,T); L^\infty_{\loc}(\Int\M;\mathbb{C}))$, that $\mathsf{b}$ is a one form with all components belonging to $\mc{G}^2((0,T); L^\infty_{\loc}(\Int\M;\mathbb{C}))$, 
	and consider the differential operator $\mathcal{T}_{\mathsf{b},\mathsf{q}}:=\partial_t^2+\Delta_g^2+\mathsf{b}\cdot\nabla_g+\mathsf{q}(t,x)$, where $\Delta_g$ is the Laplace-Beltrami operator on $\Int\M$, $\nabla_g$ the Riemannian gradient.

	Then given  $\omega$ a nonempty open set of $\M$, we have
	\begin{align*}
	\begin{cases}
		\mathcal{T}_{\mathsf{b},\mathsf{q}}  u = 0  \text{ in }(0,T)\times \Int\M \\  u \in H^1_{\loc}(0,T;H^3_{\loc}(\Int\M))  \\ u =0 \text{ in } (0,T)\times \omega    
	\end{cases}\ \Longrightarrow \ u = 0 \text{ in } (0,T)\times \Int\M .
	\end{align*}
\end{theorem}

\begin{remark}\label{rk:gevreyanalytic}
    For a Banach space $Y$, any analytic function $f: (0, T)\to Y$ also belongs to the $s$-Gevrey space $\mc{G}^s((0, T), Y)$ for $s>1$; see \cite[Definition 1.1]{FLL24}.
\end{remark}

The second block is a propagation of analyticity in time for solutions which are observed to be analytic in some region from where linear observability holds.

\begin{theorem}\label{thm:analytic-nlp}\cite[Theorem 1.7]{LL24}
    Let $(\M, g)$ be a compact connected manifold with (or without) boundary of dimension $d\leq 3$. Let $\widetilde{\omega}\subset \M$ be a nonempty open set and suppose that there exists $\omega\Subset\widetilde{\omega}$ satisfying \cref{assumOBS} for some $T_0>0$. Assume that $f\in C^\infty(\M\times \R)$ satisfies \eqref{NLP:analytic} with $f(\cdot, 0)_{|_{\partial\M}}=0$. If $T>T_0$, let $(u, \partial_t u)\in C^0([0, T], H_D^2\times L^2)$ be a solution of
    \begin{align}\label{eq:nlp-1intro}
        \left\{\begin{array}{rl}
            \partial_t^2 u+\Delta_g^2 u+\beta u+f(x, u)=0  &\ (t, x)\in [0, T]\times \Int(\M),\\
            u_{|_{\partial\M}}=\Delta u_{|_{\partial\M}}=0  &\ (t, x)\in [0, T]\times\partial\M,
            \end{array}\right.
    \end{align}
    so that for any cutoff function $\chi\in C_c^\infty(\M)$ whose support is contained in $\widetilde{\omega}$, then $t\in (0, T)\mapsto \chi u(t,\cdot)\in H^{3+\veps}(\M)\cap H_D^1(\M)$ is analytic for one $\veps>0$. Then $t\in (0, T)\mapsto \big(u(t, \cdot), \partial_t u(t,\cdot)\big)\in H^4\cap H_D^1\times H^2\cap H_D^1$ is analytic.
\end{theorem}

\begin{remark}
    In \cite{LL24} this result is not written for $f$ with an extra $x$-dependence. However, as long as $f$ is regular enough with respect to $x$, the adaptation is straightforward as we only need to ensure that \cite[Assumption 3]{LL24} is verified. The Lipschitz estimates are verified in \cref{prop:regnonlinearity} and the analyticity follows from the fact that \eqref{NLP:analytic} implies that $s\in \R\mapsto f(\cdot, s)\in H^k(\M)$ is analytic for $k\in \N$ big enough. In any case, they both rest on the nonlinear estimates \cref{app:lem:NLcomposition} as indicated in \cite[Proposition 3.4, Proposition 4.2]{LL24}.
\end{remark}

This result can be viewed as a finite-time adaptation of abstract results by Hale and Raugel \cite{HR03} on the regularity of compact attractors. Until now \cref{thm:UCP-NLP} seemed hard to obtain: although analyticity in infinite time most likely holds by Hale-Raugel results, no unique continuation for linear plates was available under non-restrictive geometric assumptions. We are in a position to prove our first main result.

\begin{proof}[Proof of \cref{thm:UCP-NLP}]
    Let $U=(u, \partial_t u)$ be a solution of \eqref{eq:nlp} in $C^0([0, T], X)$. By assumption $\omega\Subset\widetilde{\omega}$ and $\omega$ satisfies \cref{assumOBS}. Let $\chi, \chi_1\in C^{\infty}_c(\M)$ be two smooth compactly supported functions in $\widetilde{\omega}$ such that $\chi=1$ on $\omega$, $\chi_1=1$ on $\supp\chi$, $\supp\chi\Subset\supp\chi_1\Subset\widetilde{\omega}$, and $\partial_{\vec{n}}\chi=\partial_{\vec{n}}\chi_1=0$ on $\partial\M$, whose existence is guaranteed by suitable applications of \cite[Lemma A4]{LL24}. By hypothesis, we have $\partial_t u=0$ in $[0, T]\times\widetilde{\omega}$, and given that $\supp\chi_1\subset\widetilde{\omega}$, the equation satisfied by $u$ yields
    \begin{align*}
        \Delta_g^2(\chi_1 u)+\beta(\chi_1 u)=[\Delta_g^2, \chi_1]u-\chi_1 f(x, u).
    \end{align*}
    Note that $[\Delta_g^2,\chi_1]u\in H^{-1}$, provided that $u\in H_D^2$. Since $f$ is smooth, the right-hand side above belongs to $H^{-1}$ and using that $(\Delta_g^2+\beta): H_D^s\mapsto H_D^{s+4}$ is an isomorphism for any $s\in \R$, %by elliptic regularity (for instance, see \cite[Theorem 9.19]{GT01}))
    we get $\chi_1 u\in H^3(\M)\cap H_D^1$. Since $\supp\chi\Subset\supp\chi_1$, the same argument applied with $\chi$ now give us $[\Delta_g^2, \chi]u\in L^2(\M)$, which yields $\chi u\in H^{3+\veps}(\M)\cap H_D^1$ for some $\veps>0$. Hence, the map $t\in (0, T)\mapsto \chi u(t,\cdot)\in H^{3+\veps}(\M)\cap H_D^1$ is well-defined and analytic, as $u$ does not depend on $t$ on $\supp\chi\Subset\widetilde{\omega}$. Thus \cref{thm:analytic-nlp} applies and we get that $t\in (0, T)\mapsto U(t)\in H^4\cap H_D^1\times H_D^2$ is analytic.

    If we let $z:=\partial_t u$, we then see that both $t\mapsto z(t, \cdot)$ and $t\mapsto \partial_t z(t, \cdot)$ are analytic maps from $(0, T)$ into $H^4(\M)$. Therefore, $z\in H_{\loc}^1(0, T; H^4(\M))$ and it satisfies
    \begin{align*}
        \left\{\begin{array}{rl}
            \partial_t^2 z+\Delta_g^2 z+\beta z+f_s'(x, u)z=0  &\ (t, x)\in (0, T)\times \text{Int}(\M),\\
            z=0 &\ (t, x)\in (0, T)\times \omega.
        \end{array}\right.
    \end{align*}
    Since $f$ satisfies \eqref{NLP:analytic} and analyticity is stable by composition, the potential $\mathsf{q}: t\in (0, T)\mapsto \beta+f_s'(\cdot, u(t, \cdot))\in L^\infty(\M)$ is analytic in $t$ with value in $L^\infty(\M)$, and in particular, it is Gevrey-$2$ in time with value in $L^\infty(\M)$, see \cref{rk:gevreyanalytic}. As $u$ is smooth in both variables, so is $z$ and we can apply \cref{thm:UCP-linear-plate} to get that $z\equiv 0$ everywhere. 
    This in turn means that $u(t, x)=u(x)$ is constant in time, and hence $\Delta_g^2 u+\beta u+f(x, u)=0,~x\in \M$ with $u=\Delta u=0,~x\in\partial\M$. If, moreover, $sf(x, s)\geq 0$ for all $(x, s)\in \M\times\R$, multiplying the equation by $u$ and integrating by parts, we get yields
    \begin{align*}
        0\leq \int_\M \big(|\Delta_g u(x)|^2+\beta|u(x)|^2\big) dx=-\int_\M u(x)f(x, u(x)) dx\leq 0.
    \end{align*}
    By the connectedness of $\M$ and the hinged boundary conditions, we get $u\equiv 0$ everywhere. In the case $\partial\M=\emptyset$, we use that $\beta>0$ to conclude that $u\equiv 0$.
\end{proof}

\section{Stabilization and long-time dynamics}\label{sec:dyn}

\subsection{Observability inequality} By \cref{lem:ED-equiv}, the stabilization result \cref{thm:NLP-decay} follows directly once the corresponding observability inequality is established, which is the content of the following result. We employ a strategy introduced by Dehman, Lebeau and Zuazua \cite{DLZ03} in a variant used by Laurent and the author \cite[Section 3.5]{LL24} for the wave equation, we exploit the linear observability and the compactness properties of the nonlinearity in place of the usual propagation of compactness argument. In contrast with Joly and Laurent \cite{JL13}, where using the decay of the linear semigroup led them to rely on a unique continuation property in infinite time, the nonlinear observability holds in any time $T>T_0$ provided that the Schrödinger equation is $L^2$-observable in time $T_0>0$.

\begin{theorem}\label{thm:nl-observ}
 Assume $\omega$ satisfies \cref{assumOBS} for some $T_0>0$. Assume that $f$ is analytic in the second variable and defocusing, that is, satisfying \eqref{NLP:defdefocusing}. Then, for any $T>T_0$ and $E_0>0$, there exists $C>0$ so that for any $(u_0,u_1)\in H_D^2 \times L^2(\M)$, with $\norm{(u_0,u_1)}_{H_D^2\times L^2}\leq E_0$, the unique solution of \eqref{eq:nlp-stab} satisfies
    \begin{align}\label{ineq:nl-observ}
        E(U(0))\leq C\iint_{[0,T]\times\M} \gamma^2(x)|\partial_t u(t, x)|^2dxdt.
    \end{align}
\end{theorem}
\begin{proof}
    We argue by contradiction. Let $U_0^n:=(u_{0, n}, u_{1, n})\in H_D^2\times L^2(\M)$ be a sequence of initial data with $E(U_0^n)\leq E_0$, so that the unique solution $(u_n, \partial_t u_n)\in C^0([0, T], X)$ to \eqref{eq:nlp} satisfies
    \begin{align}\label{ineq:obsNLnegation}
        \iint_{[0,T]\times\M} \gamma^2(x)|\partial_t u_n(t, x)|^2dxdt<\dfrac{1}{n}E(U_0^n)\leq \dfrac{1}{n}E_0.
    \end{align}
    Up to taking a subsequence, we can assume that $(u_{0, n}, u_{1, n})$ converges weakly to $(u_0, u_1)\in H_D^2\times L^2$. By \cref{prop:wellposed} with $a=\gamma^2$ and $p=0$, $(u_n)$ is globally bounded in $C^0([0, T], H_D^2)$. In particular, up to a subsequence, we can assume that $u_n$ converges weakly to some $u \in L^2([0, T], H_D^2)$ for this topology.
    \smallskip
    \paragraph{\textit{Step 1. Identification of the weak limit}} The first step is to prove that $U=(u, \partial_t u)$ is solution of the undamped nonlinear equation with initial data $U_0=(u_0, u_1)\in H_D^2\times L^2$. The well-posedness theory allows us to consider the following nonlinear solution with initial data $U_0$,
    \begin{align*}
        V(t)=e^{t(\bA+\bB_{\gamma^2})}U_0+\int_0^t e^{(t-s)(\bA+\bB_{\gamma^2})}F(V(s))ds=:V_{\lin}(t)+V_{\Nlin}(t),
    \end{align*}
    and we want to prove that $U=V$. Since the sequence $(u_n, \partial_t u_n)$ is bounded in $C^0([0, T], H_D^2\times L^2)$, as an application of the Aubin-Lions Lemma (see \cite[Corollary 4]{S:87}), we obtain that, for any $0<\eta<1$, still up to a subsequence which we denote the same, $u_n$ converges strongly to $u$ in $C^0([0, T], H^{2-\eta}(\M))$. By leveraging that $d\leq 3$ and thus we are working on an algebra, by Sobolev embedding if we take $0<\eta<1/2$ we get that $u_n$ converges strongly to $u$ in $L^\infty([0, T]\times \M)$. Moreover, by estimates similar to the ones in \cref{prop:regnonlinearity}, we obtain
    \begin{multline*}
        \norm{f(\cdot, u_n)-f(\cdot, u)}_{L^1([0, T], L^2(\M))}\\\leq \int_0^1 \norm{f_s'(\cdot, u+\tau(u_n-u))(u_n-u)}_{L^1([0, T], L^2(\M))}d\tau
        \leq C\norm{u_n-u}_{L^\infty([0, T]\times\M)},
    \end{multline*}
    where $C>0$ is a constant that depends on $f$ and on $E_0$. The previous estimate implies that $f(u_n)\to f(u)$ strongly in $L^1([0, T], L^2(\M))$. From Duhamel's formula
    \begin{align}\label{eq:duhamelUn}
        U^n(t)=e^{t(\bA+\bB_{\gamma^2})}U_0^n+\int_0^t e^{(t-s)(\bA+\bB_{\gamma^2})}F(U^n(s))ds=U_{\lin}^n(t)+U_{\Nlin}^n(t),
    \end{align}
    with
    \begin{align}\label{eq:Nlinconvergence}
        \norm{U_{\Nlin}^n-V_{\Nlin}}_{C^0([0, T], X)}\xrightarrow[n\to\infty]{} 0.
    \end{align}
    Since the map $U_0\mapsto e^{\cdot(\bA+\bB_{\gamma^2})}U_0$ is linear continuous from $X$ into $L^2([0, T], X)$, it is continuous for the corresponding weak topology as well. Therefore, since $U_0^n$ converges weakly to $U_0$ in $X$, $e^{\cdot(\bA+\bB_{\gamma^2})}U_0^n$ converges weakly to $e^{\cdot(\bA+\bB_{\gamma^2})}U_0=V_{\lin}$ in $L^2([0, T], X)$. It follows that $U^n$ converges weakly to $V$ and consequently $U=V$.
    
    We obtained that $U=(u,\partial_t u)\in C^0([0, T], X)$ is a solution of the damped nonlinear plate \eqref{eq:nlp-stab}.
    Let $\widetilde{\omega}$ be the set of points such that $\gamma(x)>\gamma_0/2>0$, which satisfies $\omega\Subset\widetilde{\omega}$ given that $\gamma$ is smooth and $\M$ is compact. Indeed, the set $\omega_0:=\{x\in \M\ |\ \gamma(x)\geq \tfrac{2}{3}\gamma_0>0\}$ is closed (hence compact) and satisfies $\omega\subset \omega_0\subset\widetilde{\omega}$. Using \eqref{ineq:obsNLnegation} and taking weak limit in the equation satisfied by $(u_n, \partial_t u_n)$, by looking at the equation satisfied by $(u,\partial_t u)$, we get $\partial_t u=0$ on $[0, T]\times \widetilde{\omega}$. Then $U$ also solves the undamped nonlinear plate \eqref{eq:nlp} and since $f$ is defocusing, we are in position to apply the unique continuation property \cref{thm:UCP-NLP}, from which we deduce that $u\equiv 0$. In particular, $V_{\Nlin}=0$ and from \eqref{eq:Nlinconvergence} we get
    \begin{align}\label{eq:UNlin-to-0}
        \norm{U_{\Nlin}^n}_{C^0([0, T], X)}\xrightarrow[n\to\infty]{} 0.
    \end{align}
    Using the observability estimate of \cref{prop:lineardecay} and Duhamel's formula \eqref{eq:duhamelUn},
    \begin{align*}
        \norm{U_0^n}_{X}^2  &\leq C\int_0^T \norm{\gamma\partial_t u_{\lin}^n}_{L^2(\M)}^2dt\\
        &\leq 2C\int_0^T \norm{\gamma\partial_t u^n}_{L^2(\M)}^2dt+2C\int_0^T \norm{\gamma\partial_t u_{\Nlin}^n}_{L^2(\M)}^2dt.
    \end{align*}
    From \eqref{ineq:obsNLnegation} and \eqref{eq:UNlin-to-0} together, we obtain $\al_n:=\norm{U_0^n}_X\xrightarrow[n\to\infty]{} 0$.
    \smallskip
    \paragraph{\textit{Step 2. Linearization argument}} Since the sequence of initial data converges strongly to $0$, we can linearize to consider the nonlinearity as a small perturbation of the linear problem, for which we know the observability applies. Let us denote by $w_n:=u_n/\al_n$, which is a mild solution of
    \begin{align*}
        \left\{\begin{array}{rl}
            \partial_t^2 w_n+\Delta_g^2 w_n+\beta w_n+\gamma^2\partial_t w_n+\al_n^{-1}f(x, \al_n w_n)=0  &\ (t, x)\in [0, T]\times \Int(\M),\\
            w_{n|_{\partial\M}}=\Delta w_{n|_{\partial\M}}=0            &\ (t, x)\in [0, T]\times\partial\M,\\
            (w_n, \partial_t w_n)(0)=(w_{0, n}, w_{1, n}) &\ x\in \M,
        \end{array}\right.
    \end{align*}
    with $\norm{(w_{0, n}, w_{1, n})}_X=1$. By expanding $f(\cdot, s)=f(\cdot, 0)+f'(\cdot, 0)s+r(\cdot, s)$, the nonlinearity $f_n(s):=\al_n^{-1}f(\al_n s)$ can be written as $f_n(s)=f'(0)s+\al_n^{-1}r(\al_n s)$, so we set $r_n(s):=\al_n^{-1}r(\al_n s)$. Introducing
    \begin{align*}
        \widetilde{\bA}=\left(\begin{array}{cc}
        0 & I \\
        -A_0^2-\beta-f'(\cdot, 0) & -\gamma^2
    \end{array}\right)\ \text{ and }\ R_n\begin{pmatrix} u\\ v \end{pmatrix}=\begin{pmatrix}
        0 \\ -r_n(u)
    \end{pmatrix},
    \end{align*}
    we write $W^n(t)=e^{t\widetilde{\bA}}W_0^n+\int_0^t e^{(t-s)\widetilde{\bA}}R_n(W^n(s))ds=W_{\lin}^n+W_{\Nlin}^n$. Note that the defocusing assumption \eqref{NLP:defdefocusing} on $f$ implies that $f'(\cdot, 0)\geq 0$. By the observability inequality of \cref{cor:obs-plate} and \cite[Theorem 5.31]{Tre24} with potential $V=f'(\cdot, 0)$ for the linear part and \cref{app:lem:NLestimate} with Sobolev embedding for the nonlinear part, we obtain
    \begin{align*}
    \begin{array}{l}
        \norm{W_{\lin}^n}_{C^0([0, T], X)}\leq C,\\
        \norm{W_{\Nlin}^n}_{C^0([0, T], X)}\leq C\alpha_n \norm{W^n}_{C^0([0, T],X)}^2,
    \end{array}
    \end{align*}
    with $C=C(E_0, T)>0$. By putting together both bounds, we obtain
    \begin{align*}
        \norm{W^n}_{C^0([0, T], X)}\leq C+C\al_n\norm{W^n}_{C^0([0, T], X)}^2.
    \end{align*}
    A bootstrap argument \cite[Lemma 2.2]{BG99} shows that $\norm{W^n}_{C^0([0, T], X)}\leq 2C$ for $n$ large enough, which yields $\norm{W_{\Nlin}^n}_{C^0([0, T], X)}\leq C\al_n$. As we did above, by the observability estimate of \cref{cor:obs-plate} and \cite[Theorem 5.31]{Tre24} applied to the linear part $W_{\lin}^n$,
    \begin{align*}
        \norm{(w_{0, n}, w_{1, n})}_{X}^2
        \leq C\int_0^T \norm{\gamma\partial_t w^n}_{L^2(\M)}^2dt+C\int_0^T \norm{\gamma\partial_t w_{\Nlin}^n}_{L^2(\M)}^2dt
        \leq \dfrac{C}{n}+C\alpha_n^2,
    \end{align*}
    where the first term is bounded by \eqref{ineq:obsNLnegation} with the scaling $w_n=u_n/\al_n$ and $\al_n=\norm{(u_{0, n}, u_{1, n})}_{X}$, and the second term is bounded by energy estimates. Consequently $\norm{(w_{0, n}, w_{1, n})}_X\to 0$ as $n\to\infty$, which is a contradiction with the initial assumption, since 
    $1=\norm{(w_{0, n}, w_{1, n})}_{X}$ for every $n\in \N$.
\end{proof}

\subsection{Existence of a compact global attractor}\label{NLP:s:globalatt} From now on, we will assume that $f$ satisfies the asymptotic condition \eqref{cond:asymptf}, that is,
\begin{align*}
    f(\cdot, 0)_{|_{\partial\M}}=0\ \text{ and }\ sf(x, s)\geq 0\ \text{ for every }\ (x, s)\in\M\times\{|s|\geq R\}.
\end{align*} 
We recall that $(e, 0)\in H_D^2\times L^2(\M)$ is an equilibrium point of \eqref{eq:nlp} if it solves \eqref{thm:eq:nlw-equilibrium}.
We denote by $\mc{E}$ the set of equilibrium points of \eqref{eq:nlp}, which coincides with that of the damped plate \eqref{eq:nlp-stab}. The aim of this section is to prove the following result on the long-time dynamics of the damped plate equation \eqref{eq:nlp-stab}.

\begin{theorem}\label{thm:attractorA}
    Let $f$ satisfy \eqref{NLP:analytic}-\eqref{cond:asymptf}. If $\omega$ satisfies \cref{assumOBS} then the dynamical system generated by the damped plate equation \eqref{eq:nlp-stab} in $H_D^2\times L^2(\M)$ is gradient and admits a compact connected global attractor $\A$.
\end{theorem}

Let us comment on the rate at which solutions of \eqref{eq:nlp-stab} approach the set of equilibria. Due to the gradient structure, LaSalle's invariance principle \cref{app:thm:lasalle} ensures that the $\omega$-limit set of every trajectory is contained in $\mc{E}$. Suppose, in addition, that
\begin{itemize}
    \item $\mc{E}$ is finite;
    \item each equilibrium $(e,0)$ is hyperbolic, in the sense that
    \begin{align}\label{eq:hyperbolic-equil}
        (\Delta_g^2+\beta+f_s'(x,e))\,v=0,\quad v\in D(A_0^2)
        \quad\Longrightarrow\quad v=0.
    \end{align}
\end{itemize}
Then every trajectory converges to a single equilibrium at an exponential rate. Indeed, the $\omega$-limit set of a precompact forward orbit is connected, and being a connected subset of the finite set $\mc{E}$, reduces to a single point, say $\{(e, 0)\}$. The hyperbolicity condition \eqref{eq:hyperbolic-equil} is a unique continuation property for the linear plate with the indefinite-sign potential $V:=f_s'(\cdot,e)$, and it is precisely what is needed to derive the observability estimate \cref{cor:obs-plate} after recentering the equation at $e$. With this estimate at hand, we can employ the arguments of \cref{thm:nl-observ} to the recentered dynamics, with $f$ replaced by $\mathfrak{f}(x,z):=f(x,z+e)-f(x,e)$, whose linear part at $z=0$ is exactly $V$, to derive the observability estimate that yields the exponential decay towards $e$.
These are, in the abstract second-order setting and with damping acting everywhere, the hypotheses of \cite[Theorem 4.31]{CL08}, which define hyperbolicity through the same rigidity \eqref{eq:hyperbolic-equil} and give exponential decay to a hyperbolic equilibrium.

Removing these hypotheses would require different 
tools, such as the {\L}ojasiewicz-Simon inequality, available here thanks to the analyticity of $f$. This argument requires the damping to control the full gradient of the energy along trajectories; while this is effective for damping active everywhere, it is unclear under the geometrically localized damping considered here, where propagating information from $\omega$ to exert control of the energy seems to be an open question \cite[Remark 4.1]{GW13}. We do not pursue this direction, as the qualitative statement of \cref{thm:attractorA} suffices for our control strategy.

\subsubsection{The dynamical system is asymptotically smooth}  This property follows as in \cite[Proposition 4.1]{JL13}, exploiting the decay of the damped semigroup and the compactness of the nonlinearity. We recall the argument below.

\begin{proposition}\label{prop:asymptsmooth}
    Assume $\omega$ satisfies \cref{assumOBS} and let $f$ satisfy \eqref{cond:asymptf}. Let $(u_0^n, u_1^n)_{n\in\N}$ be a sequence of initial data which is bounded in $X$ and let $(u_n)_{n\in\N}$ be the corresponding solutions to the damped plate equation \eqref{eq:nlp-stab}. Let $(t_n)_{n\in\N}$ be a sequence of times such that $t_n\to+\infty$ as $n\to\infty$.

    Then, there exist subsequences $(u_{n_k})_k$ and $(t_{n_k})_k$ and a global solution $U_\infty=(u_\infty, \partial_t u_\infty)$ of \eqref{eq:nlp-stab} defined for all $t\in\R$ such that
    \begin{align*}
        (u_{n_k}, \partial_t u_{n_k})(t_{n_k}+t)\xrightarrow[k\to\infty]{} (u_\infty, \partial_t u_\infty)(t)~\text{ in }~ X.
    \end{align*}
    Moreover, $U_\infty$ is globally bounded in $X$ and the energy $t\mapsto E(U_\infty(t))$ is constant.
\end{proposition}
\begin{proof}
    By \cref{prop:wellposed} and the decay of the energy, $(U_n(s))_n$ is uniformly bounded in $X$ for $s\geq 0$. For $\sigma\in (0, 1)$ fixed, by \cref{prop:regnonlinearity}, $\norm{F(U_n(s))}_{X^\sigma}\leq M$ for every $n\in \N$ and $s\geq 0$. Writing $I_n:=\int_0^{t_n} e^{(t_n-s)(\bA+\bB_{\gamma^2})}F(U_n(s))ds$, by Duhamel's formula $U_n(t_n)=e^{t_n(\bA+\bB_{\gamma^2})}U_n(0)+I_n$. By \cref{prop:lineardecay}, the semigroup decays exponentially and thus $e^{t_n(\bA+\bB_{\gamma^2})}U_n(0)$ goes to $0$ in $X$. Moreover,
    \begin{align*}
        \norm{I_n}_{X^\sigma}\leq \int_0^{t_n}Ke^{-\kappa_0(t_n-s)}Mds\leq \dfrac{KM}{\kappa_0}\big(1-e^{-\kappa_0 t_n}\big)\leq \dfrac{KM}{\kappa_0},
    \end{align*}
    meaning that $(I_n)_n$ is relatively compact in $X$ due to the compact embedding $X^\sigma\hookrightarrow X$. By diagonal extraction, we get a subsequence $(U_{n_k}(t_{n_k}))_k$ that converges to some $U_\infty(0)$ in $X$. Let $U_\infty(t)=(u_\infty, \partial_t u_\infty)(t)$ be the maximal solution of the damped plate equation \eqref{eq:nlp-stab} with initial data $U_\infty(0)$.
    
    Let $t\in\R$ and let $k\in \N$ large enough such that $t_{n_k}+t\geq 0$. Hence, $U_{n_k}(t_{n_k}+t)$ is well-defined and, moreover, the sequence $(U_{n_k}(t_{n_k}+t))_k$ is uniformly bounded in $X$. Since our equation is wellposed, the solution is continuous with respect to the initial data. Therefore, from $U_{n_k}(t_{n_k})\to U_\infty(0)$ in $X$ as $k\to+\infty$, we obtain that $U_{n_k}(t_{n_k}+t)\to U_\infty(t)$ in $X$ as $k\to+\infty$, for each $t$ such that $U_\infty(t)$ is well-defined. Moreover, $U_\infty(t)$ is uniformly bounded as well and it can be extended to a global solution $U_\infty(t)$ for $t\in \R$.  Finally, for any $t\in \R$ it holds $E(U_{n_k}(t_{n_k}+t))-E(U_{n_k}(t_{n_k}))=-\int_{t_{n_k}}^{t_{n_k}+t} \norm{\gamma\partial_t u_{n_k}}_{L^2(\M)}^2d\tau$.
    Using that the energy functional is continuous, non-increasing, and bounded by above and below, we must have $E(U_{n_k}(t_{n_k}+t))-E(U_{n_k}(t_{n_k}))\to 0$ as $k\to \infty$, and thus in the limit $E(U_\infty(t))-E(U_\infty(0))=0$, proving that the energy is constant.
\end{proof}
\begin{remark}
    Note that for any $T>0$ we showed that $\int_0^T \norm{\gamma \partial_t u_\infty}_{L^2(\M)}^2dt=0$ and thus $\partial_t u_\infty=0$ in $(0, T)\times\omega$, which is expected for this kind of trajectories.
\end{remark}

\subsubsection{Global attractor} Here we rely on the previous results to show the existence of a compact global attractor for the damped plate equation \eqref{eq:nlp-stab}.

\begin{proof}[Proof of \cref{thm:attractorA}] Let $S_t: X\to X$ be defined by $S_tU_0=U(t)$ where $U=(u, \partial_t u)$ is the solution to \eqref{eq:nlp-stab} with initial data $U_0=(u_0, u_1)$. Let us recall that the energy functional $E\in C^0(X, \R)$ satisfies, for $s\leq t$,
\begin{align}\label{eq:E-identity}
    E(U(t))-E(U(s))=-\int_s^t \int_{\M} \gamma^2(x)|\partial_t u|^2dxd\tau.
\end{align}
From \cref{app:thm:attractor} there are four properties to be verified for the system $(X, S_t)$.

\smallskip

\paragraph{\textit{Asymptotic smoothness}} The asymptotic smoothness of $(X, S_t)$ follows by \cref{prop:asymptsmooth}, as it is precisely the Ladyzhenskaya criterion given by \cref{app:prop:ladyzhenskaya}.

\smallskip

\paragraph{\textit{The dynamical system is gradient}} We show that the energy $E\in C^0(X, \R)$ is a strict Lyapunov functional in the sense of \cref{app:def:Lyapunov}.
On the one hand, it is clear from the energy identity \eqref{eq:E-identity} that $E(U(t))\leq E(U_0)$ for every $t\geq 0$. On the other hand, let $U_0\in X$ be such that $(S_tU_0)_{t\geq 0}$ is well-defined and bounded in $X$, with $E(S_tU_0)=E(U_0)$ for every $t\geq 0$. For any $T>T_0$, from \eqref{eq:E-identity} with $s=0$ and $t=T$, it holds $\int_0^T \norm{\gamma(x)\partial_t u}_{L^2(\M)}^2d\tau=0$, hence $\partial_t u=0$ on $[0, T]\times\widetilde{\omega}$ where $\widetilde{\omega}:=\{x\in \M\ |\ \gamma(x)>\gamma_0/2\}$ is a nonempty open set satisfying $\omega\Subset\widetilde{\omega}$. Since $\omega$ satisfies \cref{assumOBS} in time $T_0<T$, we are in position to apply the unique continuation result \cref{thm:UCP-NLP}, from which follows that $e:=u(t, \cdot)=u(0, \cdot)$ for every $t\in [0, T]$, that is, $(e, 0)\in \mc{E}$ is an equilibrium point.

For precompactness of bounded positive orbits, let $U_0\in X$ be such  that $\gamma_+(U_0):=\{S_tU_0\ |\ t\geq 0\}$ is bounded in $X$. First, note that $\{S_tU_0\ |\ t\in [0, K]\}$ is compact by continuity of the map $t\mapsto S_t U_0$ due to the wellposedness of the equation. Second, for any $t_n\to+\infty$, by \cref{prop:asymptsmooth} applied to the constant sequence of initial data $U_0$, yields a subsequence $t_{n_k}\to +\infty$ such that $S_{t_{n_k}}U_0$ converges in $X$. It follows that $\gamma_+(U_0)$ is precompact in $X$.

\smallskip

\paragraph{\textit{Bounded trajectories}} 
From the proof of \cref{prop:wellposed}, using that for any $U\in X$ it holds $E(U)\geq \tfrac{1}{2}\norm{U}_X^2-C_1$ and \eqref{eq:E-identity}, for any $B\subset X$ nonempty and bounded, $\tfrac{1}{2}\norm{S_t U_0}_X^2\leq E(U_0)+C_1$ for every $U_0\in B$, hence $\bigcup_{t\geq 0} S_t B$ is bounded in $X$.
\smallskip
\paragraph{\textit{The set equilibria is bounded}} This follows by integration by parts: if $(e, 0)\in \mc{E}$, multiplying $\Delta_g^2e+\beta e+f(\cdot, e)=0$ by $e$ and using that $f$ satisfies \eqref{cond:asymptf}, 
\begin{align*}
    \norm{(e, 0)}_X^2=-\int_\M f(x, e)edx\leq -\meas(\M)\inf_{(x, s)\in\M\times\R} f(x, s)s.
\end{align*}

By \cref{app:thm:attractor}, $(X, S_t)$ admits a compact connected global attractor $\A$.
\end{proof}

\section{Semiglobal controllability}\label{sec:control}

In this section we prove \cref{thm:NLP-control}. Throughout this section, $f$ satisfies \eqref{cond:asymptf}, $\omega$ satisfies \cref{assumOBS} in time $T_0>0$ and $\omega\Subset\widetilde{\omega}$. Furthermore, we fix $\gamma\in C^\infty(\M)$ with $\gamma=1$ on $\omega$ and $\supp\gamma\subset\widetilde{\omega}$. Then the results of \cref{NLP:s:globalatt} apply to the damped equation \eqref{eq:nlp-stab} and source terms of the form $\gamma^2h$ are admissible controls for \eqref{eq:nlp-control}. As discussed in \cref{NLP:s:globalatt}, there may exist several equilibria, so that stabilization to zero cannot be expected. To deal with this issue, we follow Joly and Laurent \cite{JL14} who introduced a control technique that combines local control around equilibria with the dynamics of the damped plate equation.

\subsection{Local controllability} We now state a local controllability result near equilibrium points, whose proof is given in \cref{NLP:s:NLcontrol} below.

\begin{theorem}\label{thm:NLlocalcontrol}
    Let $T>T_0$ and let $(e, 0)\in \mc{E}$ be an equilibrium of \eqref{eq:nlp}. There exists a neighborhood $\mc{N}(e, 0)$ of $(e, 0)$ such that, for any $(u_0, u_1)$ and $(u_0^T, u_1^T)$ in $\mc{N}(e, 0)$, there exists a control $h\in L^\infty([0, 2T], L^2(\M))$ supported in $[0, 2T]\times\widetilde{\omega}$ such that the unique solution $u$ of \eqref{eq:nlp-control} satisfies $(u, \partial_t u)(2T)=(u_0^T, u_1^T)$.
\end{theorem}
Let $(e, 0)\in H_D^2\times L^2(\M)$ be an equilibrium of \eqref{eq:nlp} and set $z=u-e$. Note that $z\in C^0([0, T], H_D^2)\cap C^1([0, T], L^2(\M))$ and it satisfies
\begin{align}\label{eq:nlp-localcontrolred}
        \left\{\begin{array}{rl}
            \partial_t^2z+\Delta_g^2z+\beta z+\mathfrak{f}(x, z)=\mathbbm{1}_{\widetilde{\omega}}h  &\ (t, x)\in [0, T]\times \Int(\M),\\
            z_{|_{\partial\M}}=\Delta z_{|_{\partial\M}}=0            &\ (t, x)\in [0, T]\times\partial\M,\\
            (z, \partial_t z)(0)=(z_0, z_1) &\ x\in \M,
        \end{array}\right.
\end{align}
with $(z_0, z_1)\in H_D^2\times L^2(\M)$ and $\mathfrak{f}(x, z):=f(x, z+e)-f(x, e)$. Since $\mathfrak{f}(x, 0)=0$ for every $x\in \M$, the problem of local control towards an equilibrium in time $2T>0$ follows by the local control towards $0$ for the system \eqref{eq:nlp-localcontrolred} in time $T>0$. Indeed, by reversing time in the controlled plate equation, it is then possible to go from an equilibrium towards any desired target in time $T>0$, thus explaining that the control time spent is $2T$. Henceforth, we will focus on the (null) control problem \eqref{eq:nlp-localcontrolred} through the HUM method and a perturbation argument introduced by Zuazua \cite{Zua90}.

\subsubsection{Linear control}\label{NLP:s:Lcontrol} 
We briefly recall the HUM method in our case; see Coron \cite[Section 2.3.2]{Cor07} for a thorough exposition on the subject. Let us consider the linearization around $0$ of the control problem \eqref{eq:nlp-localcontrolred},
\begin{align}\label{eq:linprimal}
        \left\{\begin{array}{rl}
            \partial_t^2v+\Delta_g^2v+\beta v+\mathfrak{f}_s'(x, 0)v=\mathbbm{1}_{\widetilde{\omega}}h,  &\ (t, x)\in [0, T]\times \Int(\M)  \\
            v_{|_{\partial\M}}=\Delta v_{|_{\partial\M}}=0,            &\ (t, x)\in [0, T]\times\partial\M,\\
            (v, \partial_tv)(T)=(0, 0),     &\ x\in \M.
        \end{array}\right.
    \end{align}
For $(v_0, v_1)\in H_D^2\times L^2(\M)$, we say that $h\in L^2([0, T], L^2(\M))$ is a null-control for $(v_0, v_1)$ in time $T>0$ for \eqref{eq:linprimal} if $(v, \partial_t v)(0)=(v_0, v_1)$. By integration by parts we get the following duality characterization: $g$ is null-control for $(v_0, v_1)$ if and only if
    \begin{align}\label{eq:dualityplate1}
        \iint_{[0, T]\times \M} \mathbbm{1}_{\widetilde{\omega}}h\Phi dxdt=\inn{(\Phi_0, \Phi_1), (v_0, v_1)}_{X'\times X},
    \end{align}
for any $(\Phi_0, \Phi_1)\in X'=L^2(\M)\times H_D^{-2}$, where $\Phi$ solves the adjoint problem
\begin{align}\label{eq:lindual}
        \left\{\begin{array}{rl}
            \partial_t^2\Phi+\Delta_g^2\Phi+\beta \Phi+\mathfrak{f}_s'(x, 0)\Phi=0,  &\ (t, x)\in [0, T]\times \Int(\M), \\
            \Phi_{|_{\partial\M}}=\Delta \Phi_{|_{\partial\M}}=0,            &\ (t, x)\in [0, T]\times\partial\M,\\
            (\Phi, \partial_t\Phi)(0)=(\Phi_0, \Phi_1),     &\ x\in \M.
        \end{array}\right.
\end{align}
Linear wellposedness theory tells us that $(\Phi, \partial_t\Phi)\in C^0([0, T], X')$ is the only solution of \eqref{eq:lindual} such that the linear flow map $(\Phi_0, \Phi_1)\in X'\mapsto (\Phi, \partial_t\Phi)\in C^0([0, T], X')$ is bounded; see for instance \cite[Section 3]{LM72}. 

We now look for the control among localized adjoint solutions. Since $\supp\gamma\subset\widetilde{\omega}$, the choice $g:=\gamma^2\Phi\in C^0([0, T], L^2(\M))$ is admissible for \eqref{eq:linprimal}. The HUM operator $\Ld$ is the bounded linear map from $X'=L^2\times H_D^{-2}$ into $X=H_D^2\times L^2(\M)$, defined by $\Ld(\Phi_0, \Phi_1)=(v, \partial_t v)(0)$, where $v$ solves \eqref{eq:linprimal} with $g:=\gamma^2\Phi$.
By the duality identity \eqref{eq:dualityplate1}, $\Ld$ is symmetric and nonnegative. Moreover, it is an isomorphism if the following observability inequality holds: for every $(\Phi_0, \Phi_1)\in L^2\times H_D^{-2}$,
\begin{align}\label{ineq:linearobsdual}
        \norm{(\Phi_0, \Phi_1)}_{L^2\times H_D^{-2}}^2\leq C\int_0^T \norm{\gamma \Phi}_{L^2(\M)}^2dt.
\end{align}
In that case, for any initial data $(v_0, v_1)\in X$, take $(\Phi_0, \Phi_1):=\Ld^{-1}(v_0, v_1)$ and let $\Phi$ be the solution of \eqref{eq:lindual}. Then the control that steers $(v_0, v_1)$ to $(0, 0)$ in time $T$ is $g=\gamma^2\Phi$.
For more details about the HUM operator, we refer for instance to \cite{DL09} for the case of the wave equation.

Having described the HUM method, we point out right away that the exact controllability of the (compactly) perturbed system \eqref{eq:linprimal} is a consequence of \cite[Corollary 1.3]{DO18} and its verification follows as in \cref{cor:obs-plate} above.

\begin{proposition}\label{prop:Ld-isom}
    Let $\omega$ be a nonempty open set satisfying \cref{assumOBS} for some $T_0>0$. For any $T>T_0$ the HUM map $\Ld: X'\to X$ is an isomorphism.
\end{proposition}

\begin{proof}
    Set $\bA_{\mathbf{V}}=\bA+\mathbf{V},~\mathbf{V}(\psi_0, \psi_1)=(0, -V\psi_0)$ with $V(x):=\mathfrak{f}_s'(x, 0)\in C^\infty(\M)$. By compact embedding $H_D^2\to L^2(\M)$ we see that $\bA_{\boldsymbol{V}}$ is a compact perturbation of $\bA$. Let $\bB: L^2(\M)\to X$ be the bounded control operator $\mathbf{B}(g)=(0, \gamma g)$. We verify the hypothesis of \cite[Corollary 1.3]{DO18} for the pair $(\bA_{\mathbf{V}}, \mathbf{B})$ on $X$.
    \smallskip
    \paragraph{\textit{Step 1. Free observability}} Let $(\phi, \partial_t\phi)(t):=e^{t\bA}(\phi_0, \phi_1)$ solve the free plate equation with $(\phi,\partial_t\phi)(0)=(\phi_0, \phi_1)\in X'$. Setting $(z_0,z_1):=\bA^{-1}(\phi_0,\phi_1)\in X$ and $(z, \partial_t z):=e^{t\bA}(z_0, z_1)$, using that $(\beta+A_0^2)^{-1/2}: H_D^{s}\to H_D^{s+2}$ is an isomorphism for $s\in \R$, we have that
    $(\phi, \partial_t\phi)=\bA e^{t\bA}(z_0, z_1)=\tfrac{d}{dt}e^{t\bA}(z_0, z_1)$ and so $\phi=\partial_t z$. By \cref{cor:obs-plate} with $V\equiv 0$,
    \begin{align*}
        \norm{(\phi_0, \phi_1)}_{X'}^2\leq C\norm{(z_0, z_1)}_{X}^2\leq C'\int_0^T \norm{\gamma\partial_t z}_{L^2}^2dt=C'\int_0^T\norm{\gamma\phi}_{L^2}^2dt.
    \end{align*}
    By HUM, the exact controllability of $(\bA, \mathbf{B})$ in time $T>T_0$ follows.
    \smallskip
    \paragraph{\textit{Step 2. Unique continuation}} A direct computation through the duality pairing $\inn{\cdot, \cdot}_{X', X}$ gives us that $\bA_{\boldsymbol{V}}^*: D(\bA_{\boldsymbol{V}}^*)\subset X'\to X'$ with $D(\bA_{\boldsymbol{V}}^*)=D(\bA^*)=X$ is given by $\bA_{\boldsymbol{V}}^*(\psi_0, \psi_1)=(-\psi_1, (\Delta_g^2+\beta+V)\psi_0)$. As for the adjoint of the control operator $\bB^*: X'\to L^2(\M)$, we get $\bB^*(\psi_0, \psi_1)=-\gamma\psi_0$. Let $\ld\in\mathbb{C}$ and $(v_\ld, w_\ld)\in D(\bA_{\boldsymbol{V}}^*)=H_D^2\times L^2(\M)$ satisfying $\bA_{\boldsymbol{V}}^*(v_\ld, w_\ld)=\lambda(v_\ld, w_\ld)$ and $\bB^*(v_\ld, w_\ld)=-\gamma v_\ld=0$.
    The unique continuation problem for $v_\ld$ reads as
    \begin{align*}
        \left\{\begin{array}{cl}
            (\Delta_g^2+\beta+V+\ld^2)v_\ld=0,     & x\in \M,  \\
            v_\ld=0,     & x\in \omega,
        \end{array}\right.
    \end{align*}
    where $\omega\Subset\{\gamma>0\}$. As $(\Delta_g^2+\beta)v_\ld=-(V+\ld^2)v_\ld$ in $H_D^{-2}$, using that $v_\ld\in H_D^2$ and $V\in C^\infty(\M)$, the equation directly implies that $v_\ld\in H_{\loc}^4(\M)$. The unique continuation theorem \cref{thm:UCP-linear-plate} for linear (evolution) plates to obtain that $v_\ld=0$ on $\M$. From $w_\ld=-\ld v_\ld$ we obtain $w_\ld=0$ on $\M$ as well, hence $(v_\ld, w_\ld)=(0, 0)$.
    \smallskip
    \paragraph{\textit{Step 3. Conclusion}} By \cite[Corollary 1.3]{DO18}, the exact controllability of the pair $(\bA_{\boldsymbol{V}}, \bB)$ for any $T>T_0$ follows. By the HUM method, this is equivalent to the observability estimate \eqref{ineq:linearobsdual}, which implies that $\Ld: X'\to X$ is an isomorphism.
\end{proof}

\subsubsection{Nonlinear control}\label{NLP:s:NLcontrol} Here we employ a classical strategy that consists in solving the control problem as a perturbation of the corresponding linear one.

\begin{proof}[Proof of \cref{thm:NLlocalcontrol}]
    By time reversibility of the uncontrolled plate, we only need to focus on the null controllability of \eqref{eq:nlp-localcontrolred} in time $T>T_0$. We seek a control $g$ of the form $\gamma^2(x)\Phi$ where $\Phi\in C^0([0, T], L^2(\M))$ is the solution of the linearized plate equation \eqref{eq:lindual} with initial datum $(\Phi_0, \Phi_1)\in X'$, which is the adjoint system to the linearization around $0$ of the nonlinear problem
    \begin{align}\label{eq:NL-control-loc}
        \left\{\begin{array}{rl}
            \partial_t^2 z+\Delta_g^2 z+\beta z+\mathfrak{f}(x, z)=\gamma^2(x)\Phi,  &\ (t, x)\in [0, T]\times \Int(\M),\\
            z_{|_{\partial\M}}=\Delta z_{|_{\partial\M}}=0,           &\ (t, x)\in [0, T]\times\partial\M,\\
            (z, \partial_t z)(T)=(0, 0), &\ x\in \M,
        \end{array}\right.
    \end{align}
    in $C^0([0, T], X)$. Applying \cref{prop:wellposed} to $\widetilde{u}(t, \cdot)=e(\cdot)+z(T-t, \cdot)$, solution to \eqref{eq:nlp-source} with $f$, $a=0$, source $p(t,\cdot)=\gamma^2(\cdot)\Phi(T-t, \cdot)\in C^0([0, T], L^2)$ and data $(e, 0)\in X$, we deduce that there exists a unique solution $(z, \partial_t z)\in C^0([0, T], X)$ to \eqref{eq:NL-control-loc}, whose flow is Lipschitz on bounded sets and, whenever $(\Phi_0, \Phi_1)\in \mathbb{B}_{R}(X')$ with $R>0$,
    \begin{align}\label{ineq:NL-control-loc}
        \norm{(z, \partial_t z)}_{C^0([0, T], H_D^2\times L^2)}\leq C\norm{\gamma^2 \Phi}_{L^1([0, T], L^2)}\leq C\norm{(\Phi_0, \Phi_1)}_{L^2\times H_D^{-2}}.
    \end{align}
    % \smallskip
    \paragraph{\textit{Step 1. Splitting}} Let us introduce the operator
    \begin{align*}
        \begin{array}{cccl}
           \Ld_{\Nlin}:  & L^2\times H_D^{-2} & \longrightarrow & H_D^2\times L^2 \\
               & (\Phi_0, \Phi_1) & \longmapsto & \Ld_{\Nlin}(\Phi_0, \Phi_1):=(z, \partial_t z)(0).
        \end{array}
    \end{align*}
    Null controllability with control $h=\gamma^2\Phi$ then reduces to solve $\Ld_{\Nlin}(\Phi_0, \Phi_1)=(z_0, z_1)$. To perform a perturbative argument, we split $z=\Psi+v$, where $\Psi$ is solution of
    \begin{align*}
        \left\{\begin{array}{rl}
            \partial_t^2\Psi+\Delta_g^2\Psi+\beta \Psi+\mathfrak{f}_s'(x, 0)\Psi=\gamma^2(x)\Phi,  &\ (t, x)\in [0, T]\times \Int(\M),  \\
            \Psi_{|_{\partial\M}}=\Delta \Psi_{|_{\partial\M}}=0,            &\ (t, x)\in [0, T]\times\partial\M,\\
            (\Psi, \partial_t\Psi)(T)=(0, 0),     &\ x\in \M,
        \end{array}\right.
    \end{align*}
    so that $\Ld(\Phi_0, \Phi_1)=(\Psi, \partial_t\Psi)(0)$ by the HUM construction, and $v$ is solution of
    \begin{align}\label{eq:v-def-K}
        \left\{\begin{array}{rl}
            \partial_t^2 v+\Delta_g^2 v+\beta v+\mathfrak{f}_s'(x, 0)v=-\mathfrak{R}(x, z)  &\ (t, x)\in [0, T]\times \Int(\M),\\
            v_{|_{\partial\M}}=\Delta v_{|_{\partial\M}}=0            &\ (t, x)\in [0, T]\times\partial\M,\\
            (v, \partial_t v)(T)=(0, 0) &\ x\in \M,
        \end{array}\right.
    \end{align}
    with $\mathfrak{R}(x, z)=\mathfrak{f}(x, z)-\mathfrak{f}_s'(x, 0)z$. If we define $\mc{K}(\Phi_0, \Phi_1)=(v, \partial_t v)(0)$, we can write
    \begin{align*}
        \Ld_{\Nlin}(\Phi_0, \Phi_1)=\mc{K}(\Phi_0, \Phi_1)+\Ld(\Phi_0, \Phi_1),
    \end{align*}
    By \cref{prop:Ld-isom}, the HUM operator $\Ld: X'\to X$ is an isomorphism, hence $\Ld_{\Nlin}(\Phi_0, \Phi_1)=(z_0, z_1)$ is equivalent to solve $(\Phi_0, \Phi_1)=-\Ld^{-1}\mc{K}(\Phi_0, \Phi_1)+\Ld^{-1}(z_0, z_1)$, which lead us to look for a fixed point of the operator $\mc{F}: X' \longrightarrow X'$ defined by
    \begin{align*}
        \mc{F}(\Phi_0, \Phi_1)=-\Ld^{-1}\mc{K}(\Phi_0, \Phi_1)+\Ld^{-1}(z_0, z_1),
    \end{align*}
    \paragraph{\textit{Step 2. Fixed point estimates}} We first verify that the operator $\mc{F}$ is compact and continuous. This in turns boils down to prove that the operator $\mc{K}: L^2\times H_D^{-2}\to H_D^2\times L^2$ is compact and continuous when restricted to $\mathbb{B}_R(X')$ with an appropriate $R>0$. Since $e\in C^\infty(\M)$ by elliptic regularity, $\mathfrak{R}\in C^\infty(\M\times\R)$ and $\mathfrak{R}(x, 0)=\mathfrak{f}(x, 0)=0$ for $x\in \M$, hence applying \cref{prop:regnonlinearity} to $\mathfrak{R}$ instead of $f$, we get that $\mathfrak{R}(\cdot, z)\in L^1([0, T], H_D^{2\veps})$ for some $\veps>0$, and using \eqref{ineq:NL-control-loc},
    \begin{align*}
        \norm{\mathfrak{R}(\cdot, z)}_{L^1([0, T], H_D^{2\veps})}\leq C\norm{z}_{C^0([0, T], H_D^2)}\leq C\norm{(\Phi_0, \Phi_1)}_{L^2\times H_D^{-2}}
    \end{align*}
    By linear wellposedness theory, we get $(v, \partial_t v)\in C^0([0, T], H_D^{2+2\veps}\times H_D^{2\veps})$ with a bound
    \begin{align*}
        \norm{(v, \partial_t v)}_{C^0([0, T], H_D^{2+2\veps}\times H_D^{2\veps})}\leq C\norm{\mathfrak{R}(\cdot, z)}_{L^1([0, T], H_D^{2\veps})}\leq C\norm{(\Phi_0, \Phi_1)}_{L^2(\M)\times H_D^{-2}},
    \end{align*}
    which implies $\mc{K}(\Phi_0, \Phi_1)=(v, \partial_t v)(0)\in X^{\veps}$, showing that $\mc{K}$ is a compact operator. The continuity of $\mathcal{K}$ on bounded sets of $X'$, follows by composition of the following Lipschitz continuous operators: the linear flow $(\Phi_0,\Phi_1)\in X'\mapsto\Phi\in  C^0([0,T],L^2)$; the nonlinear flow $h\in L^1(0,T;L^2)\mapsto z\in C^0(0,T;L^2)$ on bounded sets by \cref{prop:wellposed}; $z\in X\mapsto \mathfrak{R}(\cdot, z)\in H_D^2$ is Lipschitz on bounded sets by \cref{prop:regnonlinearity}; and the linear flow $\mathfrak{R}(\cdot, z)\in L^1(0,T;H_D^{2\veps})\mapsto v\in C^0([0,T],X^\veps)$.
      
    It remains to prove that there exists $R>0$ so that $\mc{F}$ reproduces $\mathbb{B}_R(X')$. Since $\Ld: X'\to X$ is an isomorphism, we have
    \begin{align*}
        \norm{\mc{F}(\Phi_0, \Phi_1)}_{L^2\times H_D^{-2}}
        &\leq C\big(\norm{\mc{K}(\Phi_0, \Phi_1)}_{H_D^2\times L^2}+\norm{(z_0, z_1)}_{H_D^2\times L^2}\big)\\
        &\leq C\norm{z}_{L^\infty([0, T]\times\M)}\norm{(\Phi_0, \Phi_1)}_{L^2\times H_D^{-2}}+C\norm{(z_0, z_1)}_{H_D^2\times L^2}\\
        &\leq C\norm{(\Phi_0, \Phi_1)}_{L^2\times H_D^{-2}}^2+C\norm{(z_0, z_1)}_{H_D^2\times L^2},
    \end{align*}
    where the term $\mc{K}(\Phi_0,\Phi_1)$ is bounded above by the continuity of the flow estimate to \eqref{eq:v-def-K}, \cref{app:lem:NLestimate} applied to $\mathfrak{R}$, then \eqref{ineq:NL-control-loc} and Sobolev embedding. Thus, provided that $\norm{(z_0, z_1)}_{H_D^2\times L^2}$ is small enough, there exists $R>0$ such that $\mc{F}$ reproduces $\mathbb{B}_R(X')$ and we conclude the proof by applying Schauder's fixed point.
\end{proof}

\subsection{Traveling inside the attractor} As already pointed out, we follow the proof of \cite{JL14}. In what follows, given $U_0$ and $U_1$ in $X$, we will write $U_0\leadsto U_1$ if there exist $T^*>0$ and a control $h\in L^\infty([0, T^*], L^2(\M))$ such that the unique solution of \eqref{eq:nlp-control} with initial data $U_0$ satisfies $(u, \partial_t u)(T^*)=U_1$. Whenever needed, the control $g$ and the amount of time spent $T^*$ will be made explicit. We summarize some control properties we will use throughout the following proofs.
\begin{property}\label{property:localcontrol}
    By \cref{thm:NLlocalcontrol}, for every equilibrium $(e, 0)$, the nonlinear system \eqref{eq:nlp-control} is locally controllable in a ball $\mc{N}(e, 0)=\mathbb{B}_{r(e, 0)}(e, 0)\subset X$, with a control time $2T$, independent of $(e, 0)$, and controls supported in $[0, 2T]\times \widetilde{\omega}$.
\end{property}

\begin{property}\label{property:dampedforward}
 If $U(t)=(u, \partial_t u)(t)$ solves the damped equation \eqref{eq:nlp-stab}, by translation in time, then $U(t_0)\leadsto U(t_0+T)$ with control $h_+(t, \cdot)=-\gamma^2 \partial_t u(t_0+t, \cdot)$ acting on $[0, T]\times \widetilde{\omega}$.
\end{property}

\begin{property}\label{property:dampedbackward}
    If $U(t)=(u, \partial_t u)(t)$ solves the damped equation \eqref{eq:nlp-stab}, by reversing the flow in time, then $(u,-\partial_t u)(t_0)\leadsto(u, -\partial_t u)(t_0-T)$ with control $h_-(t, \cdot)=-\gamma^2\partial_t u(t_0-t, \cdot)$ acting on $[0, T]\times\widetilde{\omega}$.
\end{property}

The next result, called \emph{double U-turn argument} in \cite{JL14}, tell us how to travel inside the attractor around equilibrium points. The main idea is that, since the damped plate has a gradient structure, it drives any state toward an equilibrium forward in time and toward another equilibrium backward in time; alternating these dynamics and combining them with local control, we can make two \emph{U-turns} near equilibria and thereby connect any two points of the trajectory inside the compact attractor.

\begin{proposition}\label{prop:uturn}
    If $U(t)=(u, \partial_t u)(t)$ is a solution to the damped plate \eqref{eq:nlp-stab} that belongs to the compact global attractor $\A$ for each $t\in \R$, then $U(t_0)\leadsto U(t_1)$ for any $t_0$, $t_1\in \R$.
\end{proposition}
\begin{proof}
    If $t_1\geq t_0$, we just use \cref{property:dampedforward} following the damped forward flow.

    If $t_1<t_0$, by \cref{thm:attractorA} the damped plate \eqref{eq:nlp-stab} generates a gradient dynamical system. By LaSalle's invariance principle \cref{app:thm:lasalle}, the $\omega$-limit set $\omega(U(t_0))$ is nonempty and contained in $\mc{E}$. Hence, there exist $t_+>t_0$ and an equilibrium $(e_+, 0)\in \mc{E}$ such that $U(t_+)\in \mc{N}(e_+, 0)$. By hypothesis, since the whole trajectory lies in the compact attractor $\A$, the negative orbit $\gamma_-(U(t_1))$ is relatively compact, hence LaSalle's invariance principle also applies, yielding the existence of a time $t_-<t_1$ and an equilibrium $(e_-, 0)\in\mc{E}$ such that $\mc{N}(e_-, 0)$. In summary, we have $t_-<t_1<t_0<t_+$ and $(e_\pm, 0)\in \mc{E}$ with $U(t_\pm)=(u, \partial_t u)(t_\pm)\in \mc{N}(e_\pm, 0)$. Note further that $(u, -\partial_t u)(t_\pm)\in \mc{N}(e_\pm, 0)$, since we are assuming such neighborhoods to be balls centered at $(e_\pm, 0)$. The control strategy is then the following:
    \begin{align*}
        U(t_0)
        \overset{\text{(i)}}{\leadsto}(u,\partial_t u)(t_+)
        \overset{\text{(ii)}}{\leadsto}(u,-\partial_t u)(t_+)
        \overset{\text{(iii)}}{\leadsto}(u,-\partial_t u)(t_-)
        \overset{\text{(iv)}}{\leadsto}(u,\partial_t u)(t_-)
        \overset{\text{(v)}}{\leadsto}U(t_1).
    \end{align*}
    In (i),~(v) we employ \cref{property:dampedforward} following the damping forward flow; in (iii) we employ \cref{property:dampedbackward} following the backward flow to go from $t_+$ to $t_-$; in  (ii),~(iv) we employ \cref{property:localcontrol} of local control in the balls $\mc{N}(e_+, 0)$ and $\mc{N}(e_-, 0)$, respectively. Concatenating controls yields $U(t_0)\leadsto U(t_1)$.
\end{proof}

\subsubsection{Reachable set of an equilibrium} Given $U_0\in X$, let $\mc{R}(U_0)$ be the reachable set from $U_0$, defined as the set of points $U_1\in X$ such that $U_0\leadsto U_1$.

\begin{lemma}\label{lem:reachopen}
    The reachable set of an equilibrium point $(e, 0)$ is open in $X$.
\end{lemma}
\begin{proof}
    Let $U_1\in \mc{R}(e, 0)$. By definition, $(e, 0)\leadsto U_1$ with some control $h\in L^\infty([0, T], L^2(\M))$ in some time $T>0$. Holding $h$ and $T$ fixed, wellposedness theory \cref{prop:wellposed} implies that the map $\Psi_{h, T}: X\longrightarrow X$ defined by $\Psi_{h, T}(U_0)=U(T; U_0, h)$ is continuous. If $\iota(\psi_0, \psi_1)=(\psi_0, -\psi_1)$ denotes the velocity flip, then $w(t, \cdot)=u(T-t, \cdot)$ solves \eqref{eq:nlp-control} with control $\widetilde{h}(t, \cdot)=h(T-t, \cdot)$, with $\iota(u, \partial_t u)(0)=(w, \partial_t w)(T)$ and $\iota(w, \partial_t w)(0)=(u, \partial_t u)(T)$. Hence $\iota\circ\Psi_{\widetilde{h}, T}\circ\iota$ is a continuous inverse of $\Psi_{h, T}$, meaning that $\Psi_{h, T}$ is an homeomorphism in $X$, and in particular an open map. 
    Let us pick an open neighborhood $\mc{W}\subset \mc{N}(e, 0)$ of $(e, 0)$. Note then that $\mc{V}:=\Psi_{h, T}(\mc{W})$ is an open neighborhood of $\Psi_{h, T}(e, 0)=U_1$, thus for any $U_T\in\mc{V}$ there exists a unique $U^*=\Psi_{h, T}^{-1}(U_T)\in \mc{W}$, so $U^*\leadsto U_T$ with control $\widetilde{h}$. Further, by local control \cref{property:localcontrol} we have $(e, 0)\leadsto U^*$. Concatenating controls, $(e, 0)\leadsto U_T$ for every $U_T\in \mc{V}$, that is, $\mc{V}\subset\mc{R}(e,0)$ and the conclusion follows.
\end{proof}

Recall that $S_t :X\to X$ is the nonlinear flow induced by the damped plate equation, that is, $S_tU_0=U(t)$ where $U=(u,\partial_t u)$ solves \eqref{eq:nlp-stab} with data $U_0=(u_0, u_1)$.
    
\begin{lemma}\label{lem:reachclosed}
    The reachable set restricted to $\A$ of an equilibrium point $(e, 0)$, that is, $\mc{R}(e, 0)\cap \A$, is closed in $\A$.
\end{lemma}
\begin{proof}
    Let $(U_1^n)_n$ be a sequence in $\A$ such that $(e, 0)\leadsto U_1^n$ with control $h_n$ on $[0, T_n]$, and $U_1^n\to U_1$ with $U_1\in\A$. Let $U(t)=S_tU_1$ be solution of the damped plate \eqref{eq:nlp-stab} with initial condition $U_1$. Since the damped plate has a gradient structure, there exist $T>0$ large enough and an equilibrium point $(\widetilde{e}, 0)$ such that $U(T)\in \mc{N}(\widetilde{e}, 0)$. By continuity of the nonlinear semiflow $(S_t)_{t\geq 0}$ there exists $N$ large enough such that $S_TU_1^N\in \mc{N}(\widetilde{e}, 0)$.
    We now apply several successive controls:
    \begin{align*}
        (e, 0)\overset{\text{(i)}}{\leadsto} U_1^N
        \overset{\text{(ii)}}{\leadsto} S_TU_1^N
        \overset{\text{(iii)}}{\leadsto} U(T)
        \overset{\text{(iv)}}{\leadsto} U_1.
    \end{align*}
    In (i) we use the hypothesis on the sequence $(U_1^n)_n$; in (ii) we use \cref{property:dampedforward} following the forward flow by time $T>0$; in (iii) we apply local control \cref{thm:NLlocalcontrol} in $\mc{N}(\widetilde{e}, 0)$; in (iv) we use the U-turn result \cref{prop:uturn} provided that $(S_tU_1)_{t\geq 0}\subset \A$, due to the invariance of the attractor and that $U_1\in \A$. Concatenating these controls, we get that $U_1\in\mc{R}(e, 0)\cap\A$ and thus $\mc{R}(e, 0)\cap\A$ is closed in $\A$.
\end{proof}

\begin{proof}[Proof of \cref{thm:NLP-control}]
    Let us recall that $\mc{E}$ denotes the set of equilibria. Since the damped plate is gradient with nonempty compact attractor $\A$, $\mc{E}\neq \emptyset$ and $\mc{E}\subset\A$.
    For any $(e, 0)\in \mc{E}$, by \cref{lem:reachopen} and \cref{lem:reachclosed}, $\mc{R}(e, 0)\cap \A$ is both open and closed in $\A$; since the attractor $\A$ (see \cref{app:thm:attractor}) is connected and $\mc{R}(e, 0)\cap\A\neq\emptyset$, this implies that $\mc{R}(e, 0)\cap\A=\A$. Hence, by the definition of reachable set, for any two equilibria $(e_1, 0),~(e_2, 0)\in \mc{E}$ we have $(e_1, 0)\leadsto (e_2, 0)$.
    \smallskip
    \paragraph{\textit{Step 1. Control strategy}} Let $U_0$ and $U_1$ be in $\mathbb{B}$. Let $U$ and $\widetilde{U}$ be two trajectories of the damped plate \eqref{eq:nlp-stab} with $U(0)=U_0$ and $(\widetilde{u}, -\partial_t \widetilde{u})(0)=U_1$. Since the system is gradient, for $T>0$ large enough, both $U(T)$ and $\widetilde{U}(T)$ belong to neighborhoods of equilibrium points $\mc{N}(e_1, 0)$ and $\mc{N}(e_2, 0)$, respectively. Since we are assuming these neighborhoods are balls, we also have $(\widetilde{u}, -\partial_t\widetilde{u})(T)\in \mc{N}(e_2, 0)$. The control strategy is the following:
    \begin{align}\label{control-strategy}
        U_0\overset{\text{(i)}}{\leadsto} U(T)\overset{\text{(ii)}}{\leadsto} (e_1, 0)\overset{\text{(iii)}}{\leadsto} (e_2, 0)\overset{\text{(iv)}}{\leadsto} (\widetilde{u}, -\partial_t \widetilde{u})(T)\overset{\text{(v)}}{\leadsto} U_1.
    \end{align}
    In (i) we use \cref{property:dampedforward} following the forward flow by time $T>0$; in (ii), (iv) we employ \cref{property:localcontrol} of local control in $\mc{N}(e_1, 0)$ and $\mc{N}(e_2, 0)$, respectively; in (iii) we use that $\mc{R}(e_1,0)\cap\A=\mc{R}(e_2, 0)\cap\A=\A$; in (v) we use \cref{property:dampedbackward} following the backward flow. Note that all controls are supported in $\supp\gamma\subset\widetilde{\omega}$.
    \smallskip
    \paragraph{\textit{Step 2. Uniform control time}} Regarding the time of controllability, we now show that it only depends on the size of the initial data. From \cref{thm:NLlocalcontrol}, let $\tau>0$ be the uniform time of local control in any neighborhood of an equilibrium point. Also, since $\mc{E}$ is closed and contained in the compact attractor $\A$, it is a compact set itself. Let $r(e, 0)>0$ denote the radius of local controllability around an equilibrium $(e, 0)\in \mc{E}$ given by \cref{thm:NLlocalcontrol}. Let us introduce for $U\in X$, the function
    \begin{align*}
        \ld(U)=\inf\Bigg\{t\geq 0\ \big|\ S_t U\in \bigcup_{(e, 0)\in\mc{E}} \mathbb{B}_{r(e, 0)}(e, 0)\Bigg\},
    \end{align*}
    where $\mathbb{B}_{r(e, 0)}(e, 0)$ denotes the ball centered at $(e, 0)$ of radius $r(e, 0)$ in $X$. This quantity is well-defined due to the gradient structure of the damped plate. We claim that there exists $\veps>0$ such that $T_\A:=\sup_{U\in \mc{V}(\A)} \ld(U)<\infty$ where $\mc{V}(\A)$ is an $\veps$-neighborhood of $\A$. By contradiction, let us consider the sequences $T_n\to+\infty$ and $(U_0^n)_n\subset X$ with $d(U_0^n, \A)\leq 1/n$, such that $S_tU_0^n$ does not belong to any ball $\mathbb{B}_{r(e, 0)}(e, 0)$ for $t\in [0, T_n]$. Since $\A$ is compact, up to extracting a subsequence, we can assume that $U_0^n\to U_0$ with $U_0\in \A$. Given that $S_t$ is gradient, $S_TU_0\in \mathbb{B}_{r^*}(e^*, 0)$ for some $(e^*, 0)\in \mc{E}$ and $T$ large enough, and by continuity, $S_TU_0^n\in\mathbb{B}_{r^*}(e^*, 0)$ for $n$ large enough. This is a contradiction to the definition of the $T_n$'s and that $T_n\to+\infty$.

    Let $R>0$ be the radius of the ball to which the data to be controlled belongs to. We now track the time spent on each step of the control strategy \eqref{control-strategy}. By the gradient structure of the damped plate, there exists $T_a=T_a(R)$ such that $S_{T_a}\mathbb{B}_{R}(X)\subset \mc{V}(\A)$. Let us note that the set of equilibrium points $\mc{E}$ can be covered by the union of the controllability balls $\mathbb{B}_{r(e, 0)}(e, 0)$ where $(e, 0)$ runs over $\mc{E}$. By compactness, $\mc{E}$ can be covered by finitely many balls $\mathbb{B}_{r_1}(e_1, 0),\ldots, \mathbb{B}_{r_N}(e_N, 0)$, where $(e_j, 0)\in \mc{E}$. For each $(i, j)$, we have a finite time $T_{ij}$ for which $(e_i, 0)\leadsto (e_j, 0)$, and set $T_{\mathrm{eq}}:=\max_{1\leq i, j\leq N} T_{ij}$. With this at hand, once we are in $\mc{V}(\mc{A})$ we wait at most $T_\A$ to enter one of the controllability balls $\mathbb{B}_{r_j}(e_j, 0)$. The transition between the two equilibrium neighborhoods costs at most $2\tau+T_{\mathrm{eq}}$. Then, running the backward flow costs $T_a+T_\A$, which yields that the total control time is bounded by $T_{\max}(R):=2(T_a(R)+T_\A)+2\tau+T_{\mathrm{eq}}$, and it only depends on the radius of the given ball.
\end{proof}

\appendix\label{appendix}

\crefalias{section}{appendix}
\crefalias{subsection}{subappendix}

\section{Nonlinear composition estimates}\label{appendix:NL}

The following lemma establishes the regularity of the composition operator in Sobolev spaces.

\begin{lemma}\label{app:lem:NLcomposition}
    Let $N\in\N$ and $g: \M\times \R\to \R$ be a $C^{N+1}(\M\times\R, \R)$ function, with $g(\cdot, 0)=0$. If $u\in L^\infty(\M)\cap H^s(\M)$, with $s\in (0, N)$, then $g(\cdot, u)\in L^\infty(\M)\cap H^s(\M)$ and $\norm{g(\cdot, u)}_{H^s}\leq C\norm{u}_{H^s}$, where $C$ only depends on $g$ and $\norm{u}_{L^{\infty}}$.
\end{lemma}

The previous Lemma is stated in \cite{AG91} for functions in $H^s(\R^d)$ and $g\in C^{\infty}$, with no extra dependency in the $x$-variable. However, first, the same result holds for functions in $H^s(\M)$ when $\M$ is a compact manifold with boundary using the definition of the norm of $H^s(\M)$ by partition of unity and sum of the norm in $H^s(\R^d)$ of the functions in local coordinates and with extension. Second, concerning the requirement $g\in C^{N+1}$, we only notice that Meyer's multiplier Lemma \cite[Lemma 2.2]{AG91} requires estimates of the derivatives of the multiplier up to $\lfloor s\rfloor+1\leq N$. Since it is applied to $g'$ in \cite[Proposition 2.2]{AG91}, it requires $g'\in C^N$. Third, the only change to the proof when the extra $x$-dependence is added is that new (mixed) derivatives of $g$ need to be taken into account when analyzing the derivatives of the Meyer's multiplier. However, they are all uniformly bounded by a constant depending on $\norm{u}_{L^\infty}$ and on the $L^\infty$-norm of derivatives of $g$ up to order $N+1$, hence not carrying any dyadic factor depending on the derivative's order at each level. Under these considerations, the proof follows verbatim that of \cite{AG91} and we thus omit it.

\begin{lemma}\label{app:lem:NLestimate}
    Let $g: \M\times\R\to \R$ be a $C^2(\M\times\R)$ function. Then, for any $u\in L^\infty(\M)\cap L^2(\M)$ it holds
    \begin{align*}
        \norm{g(\cdot, u)-g(\cdot, 0)-g_s'(\cdot, 0)u}_{L^2(\M)}\leq C\norm{u}_{L^\infty(\M)}\norm{u}_{L^2(\M)},
    \end{align*}
    where $C$ only depends on $g$ and $\norm{u}_{L^{\infty}}$.
\end{lemma}
\begin{proof}
    By writing $g(\cdot, u)-g(\cdot, 0)-g_s'(\cdot, 0)u=\int_0^1\int_0^\tau g_s''(\cdot, \eta u)u^2d\eta d\tau$, the estimate follows by using that $g_s''$ is bounded in the compact set $\M\times\{|s|\leq \norm{u}_{L^\infty(\M)}\}$.
\end{proof}

\section{Dynamical systems}\label{appendix:DS}

Let $X$ be a Banach space. A continuous dynamical system (continuous semigroup) on $X$, is a one-parameter family of mappings $(S_t)_{t\geq 0}\subset C^0(X, X)$ such that:
\begin{itemize}
  \item $S_0=I$ and for any $t, s\in \mathbb{R}_*^+$, $S_{t+s} = S_t \circ S_s$;
%   \item for any $t\in\mathbb{R}_*^+$, $S_t \in C^0(X, X)$;
  \item for any $x_0 \in X$, the map $t \mapsto S_tx_0$ is continuous from $\mathbb{R}_*^+$ to $X$.
\end{itemize}
If $\R_*^+$ is replaced by $\R$ above, we say that $(S_t)_{t\in \R}$ is a continuous group on $X$. 

For $x_0 \in X$, we define the positive trajectory $\gamma_+(x_0) = \{ S_tx_0, t \geq 0 \}$ and the negative trajectory $\gamma_-(x_0) = \{ S_tx_0, t \leq 0 \}$. To describe the asymptotic behavior we use the concept of an $\omega$-limit set. The set $\omega(x_0) := \bigcap_{t \geq 0} \overline{\{S_s(x_0), s \geq t\}}$ is called the $\omega$-limit set of the trajectory emanating from $x_0$. This is the set of points $x \in X$ so that there exists a sequence $t_n\to+\infty$ so that $S_{t_n}x_0\to x$. Similarly, the $\alpha$-limit set is $\alpha(x_0) := \bigcap_{t \leq 0} \overline{\{S_s(x_0), s \leq t\}}$. This is the set of points $x \in X$ so that there exists a sequence $(t_n)_n$ decreasing to $-\infty$ so that $S_{t_n}x_0$ converges to $x$.

An equilibrium point is a point $e \in X$ so that $S_te = e$ for any $t \in \mathbb{R}_*^+$. We denote $\mathcal{E}$ the set of equilibrium points.

\begin{definition}
A nonempty set $\mathcal{A}\subset X$ is called a global attractor of the semigroup $(X, S_t)$ if it is closed and bounded on $X$, it is invariant, that is, $S_t\mathcal{A} = \mathcal{A}$ for every $t\geq 0$, and attracts every bounded set $B\subset X$, that is, $\sup_{x\in S_tB}\dist_X(x, \A)\to 0$ as $t\to+\infty$. If it exists, it is given by $\mathcal{A} = \{\text{bounded complete orbits of } S\}$.
\end{definition}

The existence of a compact global attractor is of particular interest, as it roughly says that the dynamics of the PDE for large times may be reduced to dynamics on a compact set, which is often finite-dimensional. Towards an existence result, we first introduce a crucial notion regarding the long-time dynamics of a dynamical system.

\begin{definition}
A dynamical system $(X, S_t)$ is asymptotically smooth iff for any nonempty, closed, bounded set $B$ such that $S_tB\subset B$ for $t>0$ there exists a compact set $K\subset B$ such that $K$ attracts $B$; that is $\sup_{x\in S_tB}\dist_X(x, K)\to 0$ as $t\to+\infty$.
\end{definition}

Asymptotic smoothness\footnote{The definition of \emph{asymptotic smoothness} is not uniform in the literature, as one may encounter a stronger notion called \emph{asymptotic compactness}. In our reference \cite{Rau02} there is no such distinction.} admits a sequential characterization, known as the Ladyzhenskaya condition, which is often easier to verify than the above definition.

\begin{proposition}\label{app:prop:ladyzhenskaya}
    If for every bounded sequence $(x_n)_n\subset X$ and every sequence $t_n\to+\infty$, the sequence $(S_{t_n}x_n)_n$ is relatively compact in $X$, then $(X, S_t)$ is asymptotically smooth.
\end{proposition}

It often occurs that a dynamical system admits some decreasing \emph{energy}.

\begin{definition}\label{app:def:Lyapunov}
A function $\Phi \in C^0(X, \mathbb{R})$ is a called a \emph{strict Lyapunov functional} for $(X, S_t)$ if
\begin{itemize}
  \item $\Phi(S_tx_0) \leq \Phi(x_0)$ for all $t \geq 0$ and $x_0 \in X$,
  \item $\Phi(S_tx_0) = \Phi(x_0)$ for all $t\geq 0$ implies that $x_0$ is an equilibrium point.
\end{itemize}
The dynamical system $(X, S_t)$ is said to be \emph{gradient} if it has a strict Lyapunov functional and each bounded positive orbit is precompact.
\end{definition}

Although the definition of a gradient system is not uniform in the literature, all of them enjoy the following property as a consequence.

\begin{theorem}[Invariance principle of LaSalle]\label{app:thm:lasalle}
Let $(X, S_t)$ be a gradient system. Then, for any $x_0 \in X$, $\omega(x_0)$ is a set of equilibrium points. If $x_0$ is so that $\gamma_-(x_0)$ is relatively compact, then $\alpha(x_0)$ is a set of equilibrium points.
\end{theorem}

Under a gradient structure, \cite[Theorem 2.26]{Rau02} simplifies to the following.

\begin{theorem}\cite[Theorem 4.6]{Rau02}\label{app:thm:attractor}
Let $(X, S_t)$ be a gradient system, so that
\begin{itemize}
  \item $(X, S_t)$ is asymptotically smooth,
  \item for any bounded set $B \subset X$, there exists $t_0 \geq 0$ such that $\bigcup_{t \geq t_0} S_tB$ is bounded,
  \item the set of equilibrium points $\mathcal{E}$ is bounded.
\end{itemize}
Then $(X, S_t)$ has a compact global attractor $\mathcal{A}$. Moreover, $\mathcal{A}$ is connected.
\end{theorem}

\bibliographystyle{alpha}
\bibliography{bibliography.bib}

\end{document}